\documentclass{article}
\usepackage[]{geometry}
\usepackage[mathcal]{euscript}
\usepackage{bm,url}                     
\usepackage{amsfonts}
\usepackage{amssymb}
\usepackage{amsmath}
\usepackage{amsthm}
\usepackage{graphicx}               
\usepackage{natbib}                 
\usepackage{colortbl}  
\usepackage{todonotes}             
\usepackage{booktabs}
\newtheorem{definition}{Definition}
\newtheorem{remark}{Remark}
\newtheorem{theorem}{Theorem}

\newtheorem{proposition}{Proposition}
\newtheorem{corollary}{Corollary}

\newcommand{\R}[1]{\mathbb{R}^{#1}}

\usepackage[english]{babel}
\usepackage{amsmath}
\usepackage{amsfonts}
\usepackage{amssymb}
\usepackage{graphicx}
\usepackage{color}
\usepackage{array}
\usepackage{mathrsfs}
\usepackage{hyperref}

\newcommand{\eps}{\varepsilon}

\newcommand{\M}{\mathcal{M}}
\newcommand{\diam}{{\rm diam}}
\newcommand{\peel}{{\rm peel}}

\definecolor{violet}{rgb}{0.7,0,0.6}
\definecolor{OliveGreen}{RGB}{85,107,47}

\begin{document}

\begin{center}
	\Large \bf  Detection of low dimensionality and data denoising via set estimation techniques\normalsize
\end{center}
\normalsize

\

\begin{center}
	Catherine Aaron$^a$, Alejandro Cholaquidis$^b$
	and Antonio Cuevas$^c$\\
	$^a$ Universit\'e Blaise-Pascal Clermont II, France\\
	$^b$ Centro de Matem\'atica, Universidad de la Rep\'ublica, Uruguay\\
	$^c$ Departamento de Matem\'aticas, Universidad Aut\'onoma de Madrid
\end{center}

\begin{abstract}
This work is closely related to the theories of set estimation and manifold estimation. 
Our object of interest is a, possibly lower-dimensional, compact set $S
\subset {\mathbb R}^d$.
The general aim is to identify (via stochastic procedures) some qualitative or quantitative features of 
$S$, of geometric or topological character. The available information is just
a random sample of points drawn on $S$. 
The term ``to identify'' means here to achieve a correct answer almost surely (a.s.) when the sample size
tends to infinity. More specifically the paper aims at giving some partial answers to the following questions: is $S$ full dimensional? Is $S$ ``close to a lower dimensional set'' $\M$? If so, can we estimate $\M$ or some functionals of $\M$ (in particular, the Minkowski content of $\M$)?  As an important auxiliary tool in the answers of these questions, a denoising procedure is proposed in order to partially remove the noise in the original data.
The theoretical results are complemented with some simulations and graphical illustrations.
\end{abstract}

\section{Introduction}\label{sec-int}
\noindent \textit{The general setup and some related literature}.  The emerging statistical field currently known as
\it manifold estimation\/ \rm (or, sometimes, \textit{statistics on manifolds}, or \it manifold learning\rm) 
is the result of the confluence of, at least, three classical theories: (a) the analysis of directional (or circular)
data \cite{mar00}, \cite{bha14} where the aims are similar to those of the classical statistics but the data are 
supposed to be drawn on the sphere or, more generally, on a lower-dimensional manifold; (b) the study of non-linear
methods of dimension reduction, \cite{del01}, \cite{has89}, aiming at recovering a lower-dimensional structure from 
random points taken around it, and (c) some techniques of stochastic geometry \cite{cha05} and set estimation 
\cite{cue10}, \cite{chola:14}, \cite{cue07} whose purpose is to estimate some relevant quantities of a set (or the 
set itself) from the information provided by a random sample whose distribution is closely related to the set.

There are also strong connections with the theories of persistent homology and computational topology, \cite{car09},
\cite{sma11}, \cite{fas14},  \cite{cav:15}.

In all these studies, from different points of view, the general aim is  similar: one wants to get  information 
(very often of geometric or topological type) on a set  from a sample of points.  To be more specific, let us mention
some  recent references on these topics, roughly grouped according the subject (the list is largely non-exhaustive):

\ \ \ \ Manifold recovery from a sample of points, \cite{gen12b,gen12c}.

\ \ \ \ Inference on dimension, \cite{fef16}, \cite{bri13}.

\ \ \ \ Estimation of measures (perimeter, surface area, curvatures), \cite{cue07}, \cite{jim11}, \cite{ber14}.

\ \ \ \ Estimation of some other relevant quantities in a manifold, \cite{sma08}, \cite{che12}.

\ \ \ \ Dimensionality reduction, \cite{gen12a}, \cite{ten00}.

\

\noindent \textit{The problems under study. The contents of the paper}. We are interested in getting some information (in particular, regarding dimensionality and Minkowski content) about a compact set ${\mathcal M}\subset {\mathbb R}^d$. While the set ${\mathcal M}$ is typically unknown, we are supposed to have a random sample of points $X_1,\ldots,X_n$ whose distribution $P_X$ has a support ``close to ${\mathcal M}$''. To be more specific, we consider two different models: 
\begin{itemize}
	\item[] \textit{The noiseless model}: the support of $P_X$ is ${\mathcal M}$ itself;  \cite{aam15}, \cite{ame02}, \cite{chola:14}, \cite{cf97}.
	\item[] \textit{The parallel (noisy) model}: the support of $P_X$ is the parallel set $S$ of points within a distance to ${\mathcal M}$ smaller than $R_1$, for some $R_1>0$, where $\M$ is a $d'$-dimensional
	set and $d'\leq d$; \cite{ber14}. Note that other different models ``with noise'' are considered in \cite{gen12a}, \cite{gen12b} and \cite{gen12c}. 
\end{itemize}


In Section \ref{sec:lower} we first develop, under the noiseless model,  an algorithmic procedure to identify, eventually, almost surely (a.s.), whether or not 
$\M$ has an empty interior; this is achieved in Theorems \ref{Th:noiseless} and \ref{th:choosern} below. A positive answer would essentially entail (under some conditions, see the beginning of Section 3) that $\M$ has a dimension smaller than  that of the ambient space.   

Then, assuming the noisy model and $\mathring{{\mathcal M}}=\emptyset$ ( where $\mathring{\mathcal{M}}$ denotes the interior of $\mathcal{M}$) Theorems \ref{Th:noisy1} (i) and \ref{thnoise1} (i) provide two methods for the estimation of the maximum level of noise $R_1$, giving also the corresponding  convergence rates. If $R_1$ is known in advance, the  remaining  results in Theorems \ref{Th:noisy1}  and \ref{thnoise1} allow us also to decide whether or not the ``inside set'' ${\mathcal M}$ has an empty interior.

The identification methods are ``algorithmic'' in the sense that they are based  on automatic procedures to perform them 
with arbitrary precision. 
This will require to impose some regularity conditions on ${\mathcal M}$ or $S$. Section \ref{sec-back} 
includes all the relevant definitions, notations and basic geometric concepts we will need.


In Section \ref{sec:alg} we consider again the noisy model where the data are drawn on the $R_1$-parallel set around a lower dimensional set ${\mathcal M}$. We propose a method  
to ``denoise'' the sample, which essentially amounts to estimate $\M$ from sample data drawn around the parallel set $S$ around $\M$.

In Section \ref{sec:minkowski} we consider the problem of estimating the $d'$-dimensional Minkowski measure 
of ${\mathcal M}$ under both the noiseless and the noisy model. We assume throughout the section that the dimension $d'$ (in Hausdorff sense, see below) of the set $\M$ is known. 

Finally, in Section \ref{sec:comput} we present some simulations and numerical illustrations.

\section{Some geometric background}\label{sec-back}

This section is devoted to make explicit the notations, and basic concepts and definitions 
(mostly of geometric character) we will need in the rest of the paper. 

\

\noindent\textit{Some notation}.
 Given a set $S\subset \mathbb{R}^d$, we will denote by
$\mathring{S}$, $\overline{S}$, $\partial S$ and $S^c$, the interior, closure, boundary and complement of $S$ 
 respectively, with respect to the usual topology of $\mathbb{R}^d$.   Let us denote $d(y,S)=\inf_{x\in S}\Vert y-x\Vert$ for $y\in {\mathbb R}^d$, where $\Vert\cdot\Vert$ stands for the Euclidean norm. 
We will also denote $\rho(S)=\sup_{x\in S}d(x,\partial S)$. Notice that $\rho(S)>0$ is equivalent 
to $\mathring{S}\neq \emptyset$.

The parallel set of $S$ of radius $\varepsilon$ will be denoted as $B(S,\eps)$, that is
$B(S,\eps)=\{y\in{\mathbb R}^d:\ \inf_{x\in S}\break \Vert y-x\Vert\leq \eps \}$.
If $A\subset\mathbb{R}^d$ is a Borel set, then $\mu_d(A)$ (sometimes just $\mu(A)$) will denote 
its Lebesgue measure.
We will denote by $\mathcal{B}(x,\varepsilon)$ (or $\mathcal{B}_d(x,\eps)$, when necessary) the closed ball
in $\mathbb{R}^d$,
of radius $\varepsilon$, centred at $x$, and $\omega_d=\mu_d(\mathcal{B}_d(x,1))$. 
Given two compact non-empty sets $A,B\subset{\mathbb R}^d$, 
the \it Hausdorff distance\/ \rm or \it Hausdorff-Pompeiu distance\/ \rm between $A$ and $C$ is defined by
\begin{equation}
d_H(A,C)=\inf\{\eps>0: \mbox{such that } A\subset B(C,\eps)\, \mbox{ and }
C\subset B(A,\eps)\}.\label{Hausdorff}
\end{equation}

\noindent\textit{Some geometric regularity conditions for sets.}
The following conditions have been used many times in set estimation topics see, 
e.g., \cite{sma08}, \cite{gen12b}, \cite{cue10} and references therein.

\begin{definition} \label{rolling} 
	Let $S\subset \mathbb{R}^d$ be a closed set. 
	The set $S$ is said to satisfy the outside $r$-rolling condition if for each boundary point $s\in \partial S$ 
	there exists some $x\in S^c$ such that $\mathcal{B}(x,r)\cap \partial S=\{s\}$.
		A compact set $S$ is said to satisfy the inside $r$-rolling condition if $\overline{S^c}$ satisfies the outside 
	$r$-rolling condition at all boundary points.
\end{definition}

\begin{definition} \label{def-rconvexity} A set $S\subset \mathbb{R}^d$ is said to be $r$-convex, for $r>0$, if 
	$S=C_r(S),$ where
	\begin{equation} \label{rhull}
	C_r(S)=\bigcap_{\big\{ \mathring{\mathcal{B}}(x,r):\ \mathring{\mathcal{B}}(x,r)\cap S=\emptyset\big\}} \Big(\mathring{\mathcal{B}}(x,r)\Big)^c,
	\end{equation}
	is the $r$-convex hull of $S$.
	When $S$ is $r$-convex, a natural estimator of $S$ from a random sample $\mathcal{X}_n$ of points 
	(drawn on a distribution with support $S$), is $C_r(\mathcal{X}_n)$.
\end{definition}

Following the notation in \cite{fed59}, let ${\rm \text{Unp}}(S)$ be the set of points $x\in \mathbb{R}^d$ 
with a unique projection on $S$.
\begin{definition} \label{def-reach} For $x\in S$, let \emph{reach}$(S,x)=\sup\{r>0:\mathring{\mathcal{B}}(x,r)\subset {\emph{Unp}}(S)\big\}$. 
The reach of $S$ is defined by $\emph{reach}(S)=\inf\big\{\emph{reach}(S,x):x\in S\big\},$ and $S$
is said to be of positive reach if $\emph{reach}(S)>0$.
\end{definition}
The study of sets with positive reach was started by \cite{fed59}; see \cite{tha08} for a survey. 
This is now a major topic in different problems of 
manifold learning or topological data analysis. See, e.g., \cite{adl16} for a recent reference.

The conditions established in Definitions \ref{rolling}, \ref{def-rconvexity} and \ref{def-reach} have an obvious mutual affinity. In fact, they are collectively referred to as ``rolling properties'' in \cite{cue12}. However, they are not equivalent: if the reach of $S$ is $r$ then $S$ is $r$-convex, which in turn implies the (outer) $r$-rolling condition. The converse implications are not true in general; see \cite{cue12} for details.

\begin{definition}\label{def-stand}  A set $S\subset \mathbb{R}^d$ is said to be standard with respect to a
Borel measure $\nu$  at  a point $x$ 
	if there exists $\lambda>0$ and $\delta>0$ such that
	\begin{equation} \label{estandar}
	\nu(\mathcal{B}(x,\eps)\cap S)\geq \delta \mu_d(\mathcal{B}(x,\eps)),\quad 0<\eps\leq \lambda.
	\end{equation}
	A set $S\subset \mathbb{R}^d$ is said to be standard if \eqref{estandar}  holds  for all $x\in S$.
\end{definition}


The following results will be useful below. The first one establishes a simple connection between standardness and the inside $r$-rolling condition. 
The second one (whose proof can be found in \cite{pateiro:09}) relates the rolling condition with the reach property.

\begin{proposition} \label{prop-stand} Let $S\subset\mathbb{R}^d$ the support of a Borel measure $\nu$, whose density $f$
	with respect to the Lebesgue measure is bounded from below by $f_0$, if $S$ satisfies  $\text{reach}(\overline{S^c})\geq r$,  then it is standard 
	with respect to $\nu$, for any $\delta\leq  f_0/3$ and 	$\lambda=r$.
\begin{proof} Let  $0<\eps\leq r$ and $x\in S$, if $d(x,\partial S)\geq r$ the result is
	obvious. Let $x\in S$ such that $d(x,\partial S)<r$.  Since $\text{reach}(\overline{S^c})\geq r$  there exists $z\in \R{d}$ such that $x\in \mathcal{B}(z,r)\subset S$. Then, for all $\eps\leq r$
	$$\nu(\mathcal{B}(x,\eps)\cap S)\geq \nu(\mathcal{B}(x,\eps)\cap \mathcal{B}(z,r))\geq f_0 \mu_d(\mathcal{B}(x,\eps)\cap \mathcal{B}(z,r))\geq  \frac{f_0}{3}\mu_d(\mathcal{B}(x,\eps)).$$\end{proof}
\end{proposition}

\begin{proposition}[Lemma 2.3 in \cite{pateiro:09}]\label{lembea} Let $S\subset\mathbb{R}^d$ be a non-empty closed set. If $S$ satisfies the inside and outside $r$-rolling condition, then $\text{reach}(\partial S)\geq r$.
	\end{proposition}

\

\noindent \textit{Some basic definitions on manifolds}\rm.  The following basic concepts are stated here for the sake of completeness and 
notational clarity. More complete information on these topics can be found, for example, in the classical
textbooks  \cite{boo75} and \cite{doc92}. See also  the book \cite{gal10} and the summary \cite[chapter 3]{zha11}.   Let us start with the classical concept of sub-manifold in ${\mathbb R}^d$ (often referred to simply as ``manifold''). Denote by ${\mathbb R}^k_+$ the half-space ${\mathbb R}^k_+=\{x\in {\mathbb R}^k:\ x_k\geq 0\}$.

\begin{definition}\label{def-manifold}
	 A  topological sub-manifold  ${\mathcal M}$ of dimension $k$ in ${\mathbb R}^d$
	is a subset of ${\mathbb R}^d$ with $k\leq d$ such  that every point in  ${\mathcal M}$ has a neighborhood 
	homeomorphic either to ${\mathbb R}^k$ or to
	${\mathbb R}^k_+$. 
	
	Those points of $\M$ having no neighborhood homeomorphic to ${\mathbb R}^k$ are called boundary points. If the boundary of $\M$ (i.e. the set of boundary points of $\M$) is empty we will say that ${\mathcal M}$ is a (sub-)manifold without boundary. 
	
	We will say that  a manifold without boundary ${\mathcal M}$ is a 
	\textit{regular $k$-surface},  or a differentiable   $k$-manifold of class $p\geq 1$, if there is  a
	family (often called \textit{atlas}) ${\mathcal V}=\{(V_\alpha,x_\alpha)\}$  of pairs $(V_\alpha,x_\alpha)$ 
	(often called \textit{parametrizations}, \textit{coordinate systems} or \textit{charts}) such that the
	$V_\alpha$ are open sets in ${\mathbb R}^k$ and the $x_\alpha:V_\alpha\rightarrow {\mathcal M}$ are
	functions of class $p$ satisfying: (i) ${\mathcal M}=\cup_\alpha x_\alpha(V_\alpha)$, (ii) every $x_\alpha$ 
	is a homeomorphism between $V_\alpha$ and $x_\alpha(V_\alpha)$ and (iii) for every $v\in V_\alpha$ 
	the differential $dx_\alpha(u):{\mathbb R}^k\rightarrow {\mathbb R}^d$ is injective. 
	
	A manifold with boundary $\mathcal{M}$ is said to be a regular $k$-surface if the set of interior points in $\mathcal{M}$ is a regular $k$-surface. 
	\end{definition}

A manifold ${\mathcal M}$ is said to be \textit{compact} when it is compact as a topological space. As a direct 
consequence of the definition of compactness, any compact  differentiable  manifold has a finite atlas. Typically, in most relevant 
cases the required atlas for a  differentiable  manifold has, at most, a denumerable set of charts.

An equivalent definition of the notion of manifold (see \citet[Def 2.1, p. 2]{doc92}) can be stated 
in terms of  \textit{parametrizations or  coordinate systems} of type $(U_\alpha,\varphi_\alpha)$ with 
$\varphi_\alpha:V_\alpha\subset {\mathcal M}\rightarrow {\mathbb R}^k$. The conditions would be completely similar
to the previous ones, except that the $\varphi_\alpha$ are defined in a reverse way to that of Definition 
\ref{def-manifold}.

In the simplest case, just one chart $x:V\rightarrow \M$ is needed. The structures defined in this way are
sometimes called \textit{planar manifolds}. 

\

\noindent \textit{Some background on geometric measure theory.} 
The important problem of defining lower-dimensional measures (surface measure, perimeter, etc.) has been 
tackled in different ways. The book by \cite{mat95} is a classical reference. We first recall the so-called 
Hausdorff measure. It is defined for any separable metric space $({\mathcal M},\rho)$. 
Given $\delta,r>0$ and $E\subset {\mathcal M}$, let
	$$
	{\mathcal H}^r_\delta(E)=\inf \left\{\sum_{j=1}^\infty (\diam (B_j))^r:\,E\subset\cup_{j=1}^\infty B_j,\ \diam (B_j)\leq \delta\right\},
	$$
	where $\diam(B)=\sup\{\rho(x,y): x,y\in B\}$, $\inf \emptyset=\infty$. 
	Now, define ${\mathcal H}^r(E)=\lim_{\delta\to 0}{\mathcal H}^r_\delta(E)$.

The set function ${\mathcal H}^r$ is an outer measure. If we restrict ${\mathcal H}^r$  to the measurable sets
(according to standard Caratheodory's definition) we get the $r$-dimensional Hausdorff measure on ${\mathcal M}$.

The Hausdorff dimension of a set $E$ is defined by
\begin{equation}
	\dim_H(E)=\inf\{r\geq 0: {\mathcal H}^r(E)=0\}=
	\sup(\{r\geq 0: {\mathcal H}^r(E)=\infty\}\cup\{0\}).\label{Def-dimhaus}
\end{equation}

It can be proved that, when ${\mathcal M}$ is a $k$-dimensional smooth manifold, $\dim_H({\mathcal M})=k$. 

\

Another popular notion to define lower-dimensional measures for the case
${\mathcal M}\subset{\mathbb R}^d$ is the \textit{Minkowski content}.   For an integer $d^\prime<d$ recall that  $\omega_{d-d^\prime}=\mu_{d-d^\prime}({\mathcal B}(0,1))$ and define the $d^\prime$-dimensional
Minkowski content of a set ${\mathcal M}$ by

\begin{equation}\label{def-min}
L_0^{d^\prime}({\mathcal M})=\lim_{\eps\to 0}\frac{\mu_d\left(B({\mathcal M},\eps)\right)}{\omega_{d-d^\prime}\eps^{d-d^\prime}},
\end{equation}

provided that this limit does exist.

In what follows we will often denote $L_0^{d^\prime}({\mathcal M})=L_0({\mathcal M})$, when the value of $d^\prime$ 
is understood. The term ``content'' is used here as a surrogate for ``measure'', as the expression \eqref{def-min} 
does not generally leads to a true (sigma-additive) measure. 

A compact set ${\mathcal M}\subset {\mathbb R}^d$ is said to be \textit{$d^\prime$-rectifiable} if there exists a 
compact set $K\subset {\mathbb R}^{d^\prime}$
and a Lipschitz function $f:{\mathbb R}^{d^\prime}\rightarrow {\mathbb R}^d$ such that ${\mathcal M}=f(K)$. 
 Theorem 3.2.39 in \cite{fed:69}  proves that for a compact $d^\prime$-rectifiable set ${\mathcal M}$,
$L_0^{d^\prime}({\mathcal M})=
{\mathcal H}^{d^\prime}(\M)$. 
More details on the relations between the rectifiability property and the structure of manifold can be found 
 in \cite{fed:69} Theorem 3.2.29.

\section{Checking closeness to lower dimensionality}\label{sec:lower}

We consider here the problem of identifying whether or not the set ${\mathcal M}\subset {\mathbb R}^d$ 
(not necessarily a manifold) has an empty interior. 

Note that, if ${\mathcal M}\subset \mathbb{R}^d$ is ``regular enough'', $\dim_H(\M)<d$ is
in fact equivalent to $\mathring{\M}=\emptyset$. 
Indeed, in general $\dim_H(\M)<d$ implies $\mathring{\M}=\emptyset$.  The converse implication is not always true, even
for sets fulfilling the property $\mathcal{H}^d(\partial \M)=0$ (see \cite{avila}). However it holds if $\M$ has 
positive reach, since in this case $\mathcal{H}^{d-1}(\partial \M)<\infty$ (see the comments after Th. 7 and inequality
(27) in \cite{amb08}). 

Also, clearly, in the case where ${\mathcal M}$ is a manifold, the fact that ${\mathcal M}$ has
an empty interior \textit{amounts to say that its dimension is smaller than that of the ambient space}.

\subsection{The noiseless model}
 
We first consider the case where the sample information follows the noiseless model explained in the Introduction,
that is,   the data $\mathcal{X}_n=\{X_1,\ldots,X_n\}$  are assumed to be an $iid$ sample of points drawn from an
unknown distribution $P_X$ with support $\M\subset \mathbb{R}^d$.  When ${\mathcal M}$ is a lower-dimensional set,
this model can be considered as an extension of the classical theory of directional (or spherical) data, in which 
the sample data are assumed to follow a distribution whose support is the unit sphere in ${\mathbb R}^d$. 
See, e.g., \cite{mar00}.

Our main tool here will be the simple \it offset\/ \rm or \textit{Devroye-Wise  estimator} (see \cite{dw:80}) given by
\begin{equation}\label{dw}
\hat{S}_{n}(r)=\bigcup_{i=1}^n \mathcal{B}(X_i,r).
\end{equation}
 
More specifically, we are especially interested in the ``boundary balls'' of $\hat{S}_{n}(r)$.

\begin{definition} Given $r>0$ let $\hat{S}_{n}(r)$  the set estimator (\ref{dw}) based on 
$\{X_1,\ldots,X_n\}$. We will say that $\mathcal{B}(X_i,r)$ is a boundary ball of $\hat{S}_{n}(r)$  if there exists
a point $y\in \partial \mathcal{B}(X_i,r)$  such that $y\in \partial \hat{S}_n(r)$. The  ``peeling'' of $\hat{S}_{n}(r)$, denoted by $\peel (\hat{S}_n(r))$, is the union of all non-boundary balls of $\hat{S}_{n}(r)$. In other words, $\peel (\hat{S}_n(r))$ is the result of removing from $\hat{S}_n(r)$ all the boundary balls. 
\end{definition}

The following theorem is the main result of this section. It relates, in statistical terms, the emptiness 
of $\mathring{\M}$ with $\peel(\hat{S}_n)$.

\begin{theorem} \label{Th:noiseless} Let $\M\subset \mathbb{R}^d$ be a compact non-empty set. Then under the model 
and notations stated in the two previous paragraphs we have,

(i) if $\mathring{\M}=\emptyset$, and $\M$ fulfills the outside rolling condition
for some $r>0$, then $\peel(\hat S_n(r'))=\emptyset$
for  any set $\hat{S}_n(r')$ of type (\ref{dw}) with $r'<r$.

(ii)  In the case $\mathring{\M}\neq \emptyset$, assume that there exists a ball $\mathcal{B}(x_0,\rho_0)\subset \mathring{\M}$ such that  $\mathcal{B}(x_0,\rho_0)$ is standard w.r.t to $P_X$, with 
 constants $\delta$ and $\lambda=\rho_0$  (see Definition \eqref{def-stand}). 
Then $\peel(\hat S_n(r_n))\neq\emptyset$ eventually, a.s., 
where $r_n$ is a radius sequence such that: $(\kappa \frac{\log(n)}{n})^{1/d}\leq r_n \leq \rho_0/2$ 
for a given
$\kappa >(\delta \omega_d)^{-1}$.

\begin{proof} (i) 
	To prove that $\peel(\hat{S}_n(r'))=\emptyset$ for all $r'<r$ it is enough to prove that for all $r'<r$ 
	and for all $i=1,\dots,n$
	there exists a point $y_i\in \partial \mathcal{B}(X_i,r')$ such that  $y_i \notin \mathcal{B}(X_j,r')$
	  
	for all $X_j\neq X_i$.
	Since $\M$ is closed and $\mathring{\M}=\emptyset$, $\partial \M=\M$. The outside rolling ball property
	implies that 
	for all $X_i\in \M$ exists $z_i\in \M^c$ such that $\mathcal{B}(z_i,r)\cap \M=\{X_i\}$. 
	Let us denote $u_i=(z_i-X_i)/r$,  then $y_i=X_i+r'u_i$  see Figure \ref{th1a}.  Clearly  
	$y_i\in\partial \mathcal{B}(X_i,r')$. From $\mathcal{B}(y_i,r')\subset\mathcal{B}(z_i,r)$ 
	and the outside rolling ball property we get that
	 $\{X_i\}\subset \mathcal{B}(y_i,r')\cap \mathcal{X}_n \subset\mathcal{B}(z_i,r)\cap \M \subset \{X_i\}$
	so that, for all $X_j\neq X_i$, $X_j\notin \mathcal{B}(y_i,r')$
	and thus, $y_i\notin\mathcal{B}(X_j,r')$.\\
		\begin{figure}[!h]
	\centering
	\includegraphics[scale=.4]{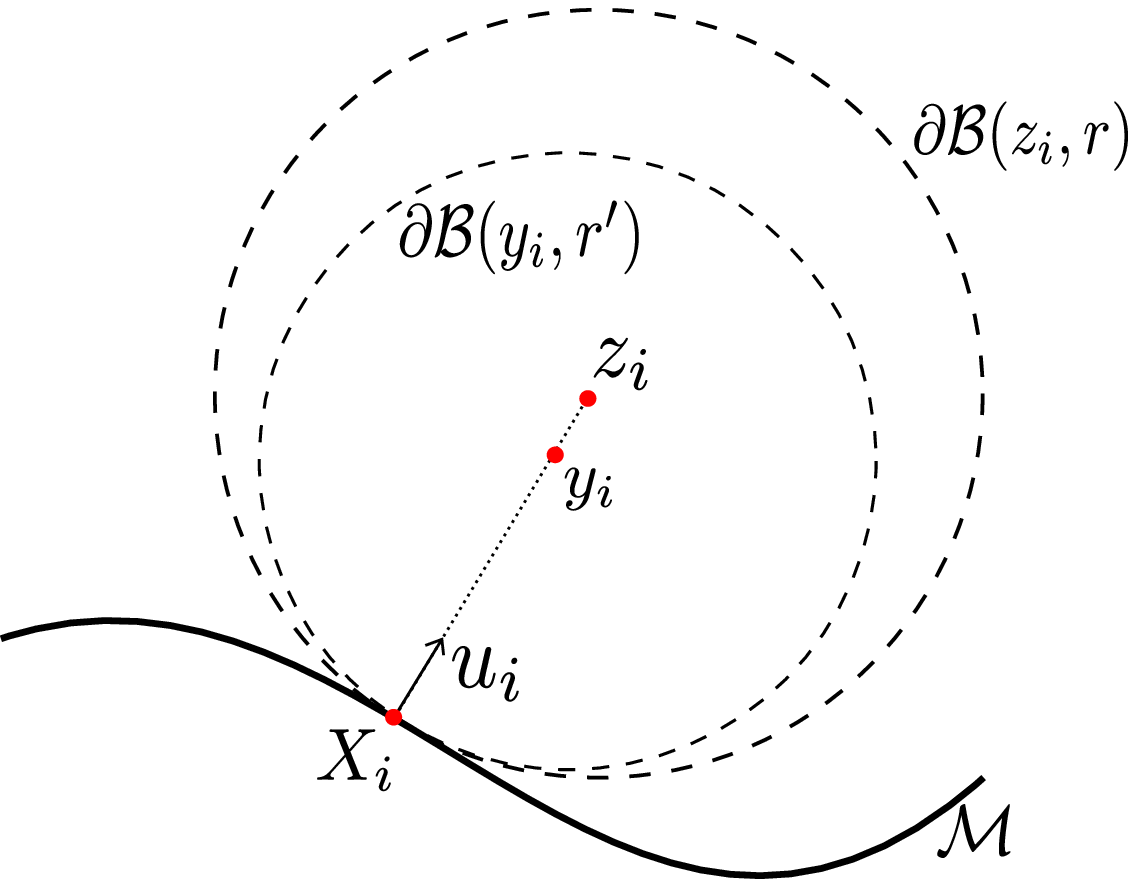}
	\caption{}%
	\label{th1a}
\end{figure}
	(ii) First  we are going to prove that
	\begin{equation}\label{th1b1}
	\begin{split}
	&\text{if }\Big(\frac{C\log(n)}{\delta \omega_d n}\Big)^{1/d}\leq r_n \leq \rho_0/2 \text{ for a given } C>1 \text{ then} \text{ eventually a.s.  for all } y\in \mathcal{B}(x_0,2r_n) \text{ we have }\\
	& \mathring{\mathcal{B}}(y,r_n)\cap \mathcal{X}_n\neq \emptyset.
	\end{split}
	\end{equation}
Consider only $n\geq 3$ and let $\varepsilon_n=(\log(n))^{-1}$, there is a positive constant $\tau_d$, such that we can cover $\mathcal{B}(x_0,2r_n)$ with $\nu_n=\tau_d\eps_n^{-d}$ balls of radius  $r_n\varepsilon_n$ centred in $\{t_1,\dots,t_{\nu_n}\}$.
	Let us define $$p_n=P_X\Big(\exists y\in \mathcal{B}(x_0,2r_n),\mathring{\mathcal{B}}(y,r_n)\cap \mathcal{X}_n=\emptyset\Big),$$
then,

\begin{equation} \label{th1b2}
p_n\leq \sum_{i=1}^{\nu_n} P_X\Big(\mathcal{B}\big(t_i,r_n(1-\varepsilon_n)\big)\cap \mathcal{X}_n=\emptyset\Big).
\end{equation}
Notice that for any given $i$,

	$$P_X\Big(\mathcal{B}\big(t_i,r_n(1-\varepsilon_n)\big)\cap \mathcal{X}_n=\emptyset\Big)=\Big(1-P_X\big(\mathcal{B}\big(t_i,r_n(1-\varepsilon_n)\big)\Big)^n. $$
Since $r_n\leq\rho_0/2$, $t_i\in \mathcal{B}(x_0,\rho_0)$, then using that $\mathcal{B}(x_0,\rho_0)$ is standard with the same $\delta$,
\begin{align*}
	P_X\Big(\mathcal{B}\big(t_i,r_n(1-\varepsilon_n)\big)\cap \mathcal{X}_n=\emptyset\Big)\leq & \Big(1-\omega_d\delta r_n^d(1-\varepsilon_n)^d\Big)^n \\
	\leq &\Big(1-C\frac{\log(n)}{n}\big(1-\varepsilon_n\big)^d\Big)^n.
	\end{align*}
Which, according to \eqref{th1b2} provides:
	$$p_n\leq \tau_d \varepsilon_n^{-d}\Big(1-C\frac{\log(n)}{n}\big(1-\varepsilon_n\big)^d\Big)^{n}\leq \tau_d\eps_n^{-d}n^{-C(1-\eps_n)^d},$$
	where we have used that $(1-x)^n\leq \exp(-nx)$. Since $C>1$, we can choose $\beta>1$ such that $p_n/n^{-\beta}\rightarrow 0$, then,  $\sum p_n<\infty$. Finally
	\eqref{th1b1} follows as a direct application of Borel Cantelli Lemma.
	Observe that \eqref{th1b1} implies that $x_0\in \hat{S}_n(r_n)$ eventually a.s.  see Figure \ref{th1b},  so there exists 
	$X_i$ such that $x_0\in \mathcal{B}(X_i,r_n)$ eventually a.s. Again by \eqref{th1b1} we get that, eventually a.s. for all
	$z\in \partial \mathcal{B}(X_i,r_n)$ there exists $X_j$ such that $z\in \mathring{\mathcal{B}}(X_j,r_n)$ and so 
	$z\notin \partial\hat{S}_n(r_n)$, which implies that, eventually a.s., $\mathcal{B}(X_i,r_n)$ is not removed by the peeling process and
	so $\peel(\hat S_n(r_n))\neq\emptyset$ eventually, a.s..\\
		\begin{figure}[!h]
		\centering
    	\includegraphics[scale=.4]{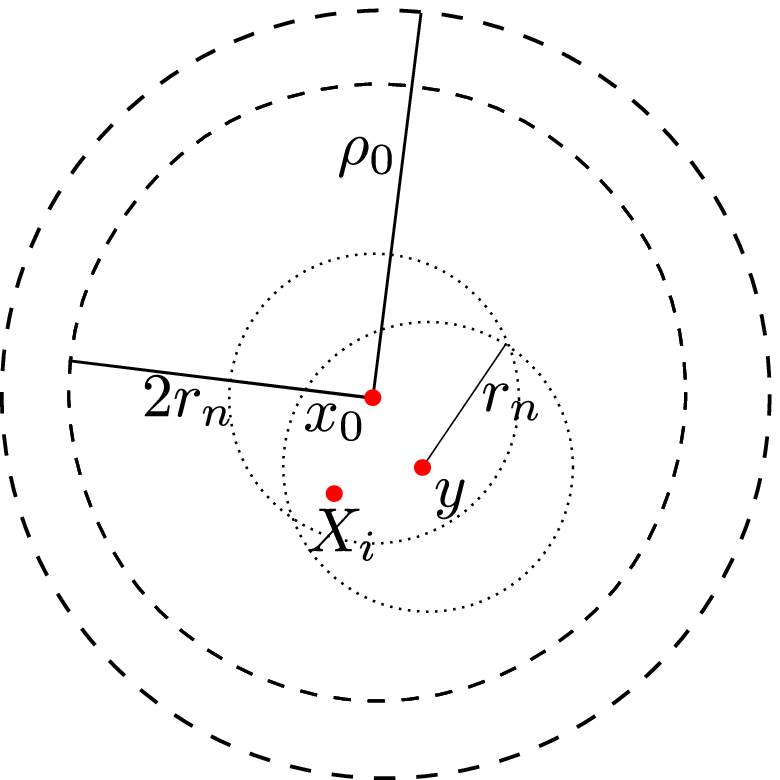}
		\caption{}%
		\label{th1b}
	\end{figure}
	\end{proof}
\end{theorem}

\begin{remark}\label{remark:r1} Some comments on Theorem \ref{Th:noiseless} are in order, regarding the intuitive meaning of the result itself, the required assumptions and the involved parameters. First note that the outside rolling condition imposed in part (i) is nothing but a geometric smoothness property ruling out the existence of very sharp inward peaks in the boundary of the set. It is close, but not equivalent, to the positive reach condition, as stated in Definition \ref{def-reach}. Clearly, the value of the parameter $r$ in Theorem \ref{Th:noiseless} is a regularity  condition on $\M$: the larger $r$, the more regular $\M$. In general,  if we want to obtain, using statistical methods, some meaningful results on the dimensionality or the interior of $\M$, we will need to impose some regularity property. The advantage of the rolling condition is its simple intuitive, almost ``visual'', interpretation. See \cite{wal99} and \cite{cue12} for further insights on the rolling condition and related properties. 
	
	Regarding part (ii): if  $\mathring{\M}\neq\emptyset$ there must be some ball ${\mathcal B}(x_0,\rho_0)$ included in $\mathring{\M}\neq\emptyset$. The standardness assumption imposed in the theorem, only asks that the probability $P_X$ is not ``too far from uniformity'' on that ball. To be more specific, the probability of the intersection with ${\mathcal B}(x_0,\rho_0)$ of any small enough ball $B$  centered at a point of ${\mathcal B}(x_0,\rho_0)$ must be at most $\delta$ times the volume of $B$. Observe that this mild condition holds, in particular, whenever $P_X$ has a density $f$ bounded 
	from below by a positive constant. More insights on the meaning and use of this standardness property   can be found, for example,  in \cite{cf97} and \cite{rin10}.
	
	Finally, about the interpretation of parts (i) and (ii) in the theorem: statement (i) is simple. It just establishes that the property $\mathring \M=\emptyset$ can be identified, with probability one, whatever the simple size using the offset estimator \eqref{dw} with any radius smaller that the assumed rolling parameter $r$.  As for part (ii), let us note that the only relevant parameter is the standardness  constant $\delta$.  A conservative choice of $\delta$ would also do the job asymptotically. In this case,  the identification of 
	$\mathring{\M}\neq\emptyset$ is done asymptotically (eventually, a.s.) by taking the offset estimator with balls of radii $r_n$ depending only on $\delta$ and $n$. The order $(\log n/n)^{1/d}$ of such balls appears typically in the convergence rates of many set estimators (see \cite{cf97}, \cite{rod07}) as well as in the theory of multivariate spacings, \cite{jan87}. 
	
	Hence, in summary, the method to identify whether or not $\mathring{\M}=\emptyset$ is completely ``algorithmic'' and works, under some regularity conditions on $\M$, with probability one. While the situation $\mathring{\M}=\emptyset$ is easy to identify, the identification of  $\mathring{\M}\neq\emptyset$ only works asymptotically. 
\end{remark}

\

\noindent \textit{The manifold case}.
If $\M$ is assumed to be a manifold, then, under some mild additional assumptions, the identification of low dimensionality can be done in a completely automatic (data-driven) way, with no resort to extra parameters. In other words, the radius of the balls in the auxiliary Devroye-Wise estimator  can be chosen as a function of the data in such a way that it is (asymptotically) small enough to identify the situation $\mathring \M=\emptyset$ and large enough to eventually detect $\mathring \M=\emptyset$, when this is the case.

 \begin{theorem}\label{th:choosern}
{Let $\M$ be a $d'$-dimensional compact manifold in ${\mathbb R}^d$.
Suppose that the sample points $X_1,\ldots,X_n$ are drawn from a probability measure $P_X$  with support $\M$ 
which  has a density $f$, with respect the $d'$-dimensional Hausdorff measure on $\M$, continuous on $\M$ such that
$f_0=\min_{x\in \M}f(x)>0$. Let us define, for any $\beta>6^{1/d}$, $r_n=\beta\max_i\min_{j\neq i} \|X_j-X_i\|$.
Then,
\begin{itemize}
	\item[$i)$] if  $d'=d$   and $\partial \mathcal{M}$ is a $\mathcal{C}^2$  manifold then $\peel(\hat S_n(r_n))\neq\emptyset$ eventually, a.s.. 
	\item[$ii)$]  if  $d'<d$ and $\M$ is a $\mathcal{C}^2$  manifold without boundary, then  $\peel(\hat S_n(r_n))=\emptyset$ eventually, a.s.
\end{itemize} }
\end{theorem}
\begin{proof}
{
	\begin{itemize}
	\item[$i)$]  We will use Theorem \ref{Th:noiseless} (ii). In order to do that, we will prove first that the set is standard.  
	 As $d'=d$ then  $\partial \mathcal{M}$ is a $\mathcal{C}^2$ a compact  $(d-1)$-manifold.
	 Then we can use the following result, due to \cite{wal99}. \\

\
Theorem \cite[Th.1]{wal99}.- Let $S\subset \mathbb{R}^d$ be a compact path-connected set with $\mathring S\neq\emptyset$ and let $r_0>0$. Then, the following conditions are equivalent
	\begin{itemize}
		\item[1] A ball of radius $r$ rolls freely inside $S$ and inside $\overline{S^c}$ for all $0\leq r\leq r_0$. 
		\item[2] $\partial S$ is a $(d-1)$-dimensional $C^1$ submanifold in $\mathbb{R}^d$ with the outward pointing unit normal vector $n(s)$ at $s\in \partial S$ satisfying the Lipschitz condition
		$$\|n(s)-n(t)\|\leq \frac{1}{r_0}\|s-t\|\quad \text{ for all }s,t\in \partial S$$ 
	\end{itemize}

	In fact, the author points out that the result is also valid if the condition of path-connected is dropped and we only assume that every path connected component of $\M$ has non-empty interior. Hence, note that this result can be applied in our case for $S=\M$ since the ${\mathcal C}^2$ assumption on the compact hypersurface $\partial\M$ implies the Lipschitz condition for the outward normal vector and the assumption $\mathring\M_1\neq \emptyset$ for every path-connected component $\M_1$ of $\M$ is guaranteed from the fact that every point in $\M$ has a neighborhood homeomorphic to an open set in ${\mathbb R}_+^d$.  Thus, we may use the result 2 $\Rightarrow$ 1 in the above theorem to conclude that $\M$ fulfills both the inside and outside rolling ball property for a small enough radius $r>0$.  Then by Proposition \ref{lembea} $\text{reach}(\partial \M)\geq r$. So, by Proposition \ref{prop-stand}, $\M$ satisfies  the 
	standardness condition established in Definition \ref{def-stand} with $\nu=P_X$, $\delta= f_0/3$ and 
	$\lambda<r$.	 Now, in order to prove that $r_n$ fulfils all the conditions in Theorem \ref{Th:noiseless} (ii)  observe that in the full-dimensional case $d'=d$ the intrinsic volume in $\M$ 
	coincides with the restricted Lebesgue measure; see \cite[Prop. 12.6]{tay06}. As a consequence, $f$ is equal to 
	the density of $P_X$ w.r.t. the Lebesgue measure restricted to $\M$. Let us denote 
	$f_1=\min_{x\in \partial \M}f(x)$.   Note that $r_n/\beta$ is in fact the ``connectivity statistic", that is the minimum value of $r$ such that $\cup_i{\mathcal B}(X_i,r)$ is a connected set. Then, as $f$ is continuous and bounded below from zero on the compact set $\M$ with smooth boundary we are in the assumptions of Theorem 1.1 in \cite{penrose:99} so that, using this result we can conclude that, with probability one, we have,
	$$\frac{nr_n^d\omega_d}{\log(n)\beta^d}\rightarrow \max\Big\{\frac{1}{f_0},\frac{2(d-1)}{df_1}\Big\}\geq \frac{1}{f_0}.$$
	Then for $n$ large enough,
	$$r_n\geq \Big(\frac{\log(n)}{n}\frac{\beta^d}{\omega_d2f_0}\Big)^{1/d},$$
	now if we denote $\kappa=\beta^{d}/(\omega_d2f_0)$, it fulfills that $\kappa>(\delta \omega_d)^{-1}$, so we are in 
	the hypotheses of Theorem \ref{Th:noiseless} (ii) and then we can conclude $\peel(\hat S_n(r_n))\neq\emptyset$ 
	eventually, with probability 1.
	
	\item[$ii)$] 
	 Notice that we can use Theorem  \ref{Th:noiseless} (ii) indeed, as $\M$ is a  $\mathcal{C}^2$ 
	compact manifold of ${\mathbb R}^d$
	 by \cite[Prop. 14]{tha08} it has a positive reach and, thus, it satisfies the outside rolling ball condition (for some radius $r>0$).
	 Then it remains to be proved that  $r_n\leq r$ for $n$ large enough. 
	 Let us endow $\M$ with the standard Riemannian structure, where a local metric is defined on every tangent space just by restricting on it the standard inner product on ${\mathbb R}^d$.  Under smoothness assumptions, the Riemannian measure induced by such a metric on the manifold $\M$ agrees with the $d'$-dimensional Hausdorff measure on $\M$ (this  is just a particular case of the Area Formula; see \cite[3.2.46]{fed:69}). So we may use Theorem 5.1 in \cite{penrose:99}. As a consequence of that result
	 \begin{equation}
	 \max_i\min_{j\neq i} \gamma(X_i,X_j)={\mathcal O}\left(\left(\frac{\log n}{n}\right)^{1/d'}\right),\ \mbox{a.s.},\label{eq:pen51}
	 \end{equation}
where $\gamma$ denotes the geodesic distance on $\M$ associated with the Riemannian structure. Now, since the Euclidean distance is smaller than the geodesic distance, we have for all $i,j$,\break $\|X_j-X_i\|\leq \gamma(X_i,X_j)$ and
	 $\min_j\gamma(X_i,X_j)=\gamma(X_i,X_{i'})\geq \|X_i-X_{i'}\|\geq\break \min_j \|X_i-X_j\|$ and finally 
	 $\max_i \min_{j\neq i} \gamma(X_i,X_j) \geq \max_i \min_{j\neq i} \|X_i-X_j\|$.
	Finally from \eqref{eq:pen51} we have $\max_i\min_{j\neq i} \|X_j-X_i\|\stackrel{a.s.}{\longrightarrow}0$, which concludes
	the proof.
	\end{itemize}}
\end{proof}

\

\subsection{The case of noisy data: the ``parallel'' model}\label{subsec:noisy}

The following two theorems are meaningful in at least two ways. On the one hand, if we know the amount of noise ($R_1$ in the notation introduced before), 
these results can be used to detect whether or not the support $\M$ of the original sample
is full dimensional  
(see \eqref{ineqth2} and \eqref{ineqth2bis}).

On the other hand, in the lower dimensional setting, they give an easy-to-implement way to estimate $R_1$ { 
(see \eqref{ineqth} and \eqref{ineqthbis})}.

Observe that when $\mathring{\M}=\emptyset$, then $R_1=\max_{x\in S} d(x,\partial S)$. 
If {{$\widehat{\partial S}_n$}} denotes a consistent estimator of $\partial B(\M,R_1)$, a natural plug-in estimator for 
$R_1$ is $\max_{Y_i\in \mathcal{Y}_n} d(\mathcal{Y}_n,\widehat{\partial S}_n)$.

 In Theorem \ref{Th:noisy1} $\widehat{\partial S}_n$ is constructed  in terms of the set of the centers of the boundary balls, while in Theorem 
\ref{thnoise1} we 
use the boundary of the $r$-convex hull. The second theorem is stronger than the first one in several aspects: the parameter 
choice is easier and the convergence rate is better (and does not depend on the parameter). The price to pay is computational
since the corresponding statistic is much more difficult to implement; see Section \ref{sec:comput}.

\begin{theorem} \label{Th:noisy1} Let $\M\subset \mathbb{R}^d$ be a compact set such that $\emph{reach}(\M)=R_0>0$.
 Let $R_1$ be a constant with $0<R_1<R_0$ and  let $\mathcal{Y}_n=\{Y_1,\dots,Y_n\}$ be an iid sample of a distribution $P_Y$ with support $S=B(\M,R_1)$, 
absolutely continuous with respect to the Lebesgue measure, whose density $f$ is bounded from below  
by $f_0>0$.
Let $\eps_n=c(\log(n)/n)^{1/d}$, with  $c>(6/(f_0\omega_d))^{1/d}$,   
 and let us denote $\hat{R}_n=\max_{Y_i\in \mathcal{Y}_n}\min_{j\in I_{bb}}\|Y_i-Y_j\|$ 
where $I_{bb}=\{j:\mathcal{B}(Y_j,\eps_n) \text{ is a boundary ball}\}$.
\begin{itemize}
	\item[i)] if $\mathring{\M}=\emptyset$ then, with probability one,
      \begin{equation} \label{ineqth}
      \left|\hat{R}_n-R_1\right|\leq  2  \eps_n \ \text{ for } n \text{ large enough},
      \end{equation}

      \item[ii)] if $\mathring{\M}\neq\emptyset$ then there exists $C>0$ such that, with probability one
\begin{equation} \label{ineqth2}
\left|\hat{R}_n-R_1\right|>C \ \text{ for } n \text{ large enough}.
\end{equation}
\end{itemize} 
\begin{proof} 
\begin{itemize}
\item[$i)$]

 Observe, that, since $\mathring{\mathcal{M}}=\emptyset$, $R_1=\max_{x\in S} d(x,\partial S)$. Then, the proposed estimator $\hat R_n$ is quite natural: roughly speaking, we may consider that the set of centres of the boundary balls is an estimator of the boundary of $S$ so that the maximum distance from the sample points to these centres is a natural estimator of the  parameter $R_1$ that measures the ``thickness'' of $S$.  We will now use 
Corollary 4.9 in \cite{fed59}; this result establishes that for $r>0$, the $r$-parallel set of a non-empty closed set $A$ fulfills $\mbox{reach}(B(A,r))\geq \mbox{reach}(A)-r$. Also, $\mbox{reach}\{x:d(x,A)\geq r\}\geq r$. Then, in our case, for $S=B(\M,R_1)$, this result yields  $\text{reach}(S)\geq R_0-R_1>0$ and $\text{reach}(\overline{S^c})\geq R_1$.  By Proposition 1 and 2 in \cite{cue12} $S$ fulfils the inner and outer rolling condition.

Another consequence of the positive reach of $S$ is that it has a Lebesgue null boundary and thus, 
with probability one for all $i$, $Y_i\in \mathring{S}$ and then, with probability one
\begin{equation}\label{1aster}
 \hat{S}_n(\eps_n)\subset B(\mathring{S},\eps_n).
\end{equation}
 Since $\text{reach}(\overline{S^c})>0$,  by Proposition \ref{prop-stand} $S$ is standard with respect to 
$P_X$ for any constant $\delta<f_0/3$ (see Definition \ref{def-stand}). 

\

 Now, we will use Theorem 4 and Proposition 1  in \cite{crc:04}; according to these result, if $S$ is partially expandable and it is standard with respect to $P_X$ (both conditions are satisfied in our case)  we have for  large enough $n$, with
probability one, 
\begin{equation}\label{2aster}
 S\subset\hat{S}_n(\eps_n),
\end{equation}
for a choice of $\eps_n$ as that indicated in the above statement of the theorem.

For all $x\in S$ let us consider $z\in \partial S$ a point such that $\|x-z\|=d(x,\partial S)$ and
$t=z+\eps_n\eta$ where $\eta=\eta(z)$ is a normal vector to $\partial S$ at $z$ that points outside
$S$ ($\eta$ can be defined according to Definition 4.4 and Theorem 4.8 (12) in \cite{fed59}). 
 Notice that the metric projection of $t$ on $S$ is $z$ 
thus $d(t,S)=\eps_n$ so, according to \eqref{1aster},
with probability one $t\notin \hat{S}_n(\eps_n)$. The point $z$ belongs to $S$ so, by \eqref{2aster}, 
with probability one for $n$ large enough  $z\in \hat{S}_n(\eps_n)$. We thus conclude
$[t,z]\cap \partial \hat{S}_n(\eps_n)\neq \emptyset$, with probability one, for $n$ large enough. Let then consider 
$y\in [t,z]\cap \partial \hat{S}_n(\eps_n)$, as $y\in \partial \hat{S}_n(\eps_n)$ there exists 
$i\in I_{bb}$ such that $y\in \partial \mathcal{B}(Y_i,\eps_n)$ and, as $y\in [t,z]$,  $\|y-z\|\leq \eps_n$ 
thus $\|x-Y_i\|\leq \|x-z\|+\|z-y\|+\|y-Y_i\|\leq d(x,\partial S)+2\eps_n$. To summarize we just have proved that: 
for all $x\in S$ there exits $i\in I_{bb}$ such that $\|x-Y_i\|\leq d(x,\partial S)+2\eps_n$ thus
for all $x\in S$ :  $\min_{i\in I_{bb}} \|x-Y_i\|\leq d(x,\partial S)+2\eps_n$. To conclude
 $\max_j \min_{i\in I_{bb}} \|Y_j-Y_i\|\leq \max_j d(Y_j,\partial S)+2\eps_n\leq  \max_{x\in S}d(x,\partial S)+2\eps_n=R_1+2\eps_n$ 
 (with probability one for $n$ large enough).

The reverse inequality is easier to prove, let us consider $x_0\in S$ such that $d(x_0,\partial S)=R_1$, 
notice that, by \eqref{2aster} (with probability one for $n$ large enough) there exists $i_0$ such that \break
$\|x_0-Y_{i_0}\|\leq \eps_n$. By triangular inequality $\mathcal{B}(Y_{i_0},R_1-\eps_n)\subset S$ and by
\eqref{2aster}  we also have $\mathcal{B}(Y_{i_0},R_1-\eps_n)\subset \hat{S}_n(\eps_n)$ thus $\min_{i\in I_{bb}}\{\|Y_{i_0}-Y_i\|\}\geq R_1-2\eps_n$.
Then we have proved $\max_j \min_{i\in I_{bb}}\{\|Y_{i}-Y_j\|\}\geq R_1-2\eps_n$. This concludes the proof of \eqref{ineqth}.

	\item[$ii)$] Observe that to prove $i)$ we proved that $|\hat{R}_n-\max_{x\in S}d(x,\partial S)|< 2  \eps_n$. 
	Then, with probability one, for $n$ large enough, $|\hat{R}_n-R_1|>|c_1-R_1|/2=C>0$, where $c_1=\max_{x\in \partial S}d(x,\partial S)$.
	\end{itemize}
\end{proof}
\end{theorem}

\begin{theorem} \label{thnoise1}
{ Let $\M\subset \mathbb{R}^d$ be a compact set} such that $\emph{reach}(\M)=R_0>0$. Suppose that the sample $\mathcal{Y}_n=\{Y_1,\ldots,Y_n\}$ 
has a distribution with support $S=B({\mathcal M},R_1)$ for some $R_1<R_0$ with a density bounded from below by a constant $f_0>0$.  
Let us denote $\tilde{R}_n= \max_{i} d(Y_i, \partial C_r(\mathcal{Y}_n))$ where $C_r(\mathcal{Y}_n)$ denotes the $r$-convex hull of the sample, as defined in \eqref{rhull} 
for $r\leq \min(R_1,R_0-R_1)$.
\begin{itemize}
	\item[i)] If $\mathring{\M}=\emptyset$ { and for some $d^\prime<d$ $\M$ has a finite, strictly positive $d^\prime$-dimensional Minkowski content,  then, with probability one,
      \begin{equation} \label{ineqthbis}
      \left|\tilde{R}_n-R_1\right|= \mathcal{O}\big(\log(n)/n\big)^{\min(1/(d-d'),2/(d+1))},\ 
      \end{equation}}
      \item[ii)] if $\mathring{\M}\neq\emptyset$, then there exists $C>0$ such that, with probability one
\begin{equation} \label{ineqth2bis}
\left|\tilde{R}_n-R_1\right|>C \ \text{ for } n \text{ large enough}.
\end{equation}
\end{itemize} 

\begin{proof}
	
	Again, as shown in the proof of Theorem \ref{Th:noisy1}, 
	$\mbox{reach}(B({\mathcal M},R_1))\geq \mbox{reach}({\mathcal M})-R_1=R_0-R_1$; also $\mbox{reach}(\overline{B({\mathcal M},R_1)^c})\geq R_1$.   We now use Proposition 1 in \cite{cue12}; this result establishes that $\mbox{reach}(S)\geq r$ implies that $S$ is $r$-convex. According to this result we may conclude that 
	$B({\mathcal M},R_1)$ and $\overline{B({\mathcal M},R_1)^c}$ are both $r$-convex
	for $r=\min(R_1,R_0-R_1)>0$. 
	Note, in addition, that 
	by construction of $S=B({\mathcal M},R_1)$ we have that $\mathring{S_i}\neq \emptyset$ for every path-connected  component 
	$S_i\subset S$. So, we can use Theorem 3 in \cite{rod07}   (which establishes the rates of convergence in the estimation of an $r$-convex set using the $r$-convex hull of the sample)  to conclude 
	\begin{equation} \label{eqthrhull}
	d_H\big(\partial C_r(\mathcal{Y}_n),\partial S\big)=\mathcal{O}\big((\log(n) /n)^{2/(d+1)}\big), \text{ a.s.}
	\end{equation}

	Let us now prove that, with probability one, for $n$ large enough,
	\begin{equation}\label{th4inclu}
	B\big(\M,R_1-d_H(\partial C_r(\mathcal{Y}_n),\partial S)\big)\subset C_r(\mathcal{Y}_n).
	\end{equation}
	Proceeding by contradiction, let $x_n\in B\big(\M,R_1-d_H(\partial C_r(\mathcal{Y}_n),\partial S)\big)$ such that $x_n\notin C_r(\mathcal{Y}_n)$, let 
	$y_n$ be the projection of $x_n$ onto $\M$. It is easy to see that, for $n$ large enough, with probability one, $\M \subset C_r(\mathcal{Y}_n)$ then $y_n\in C_r(\mathcal{Y}_n)$. 
	Observe that,  from the definition of parallel set, 
	\begin{equation} \label{th4inclu2}
	B\big(\partial S,d_H(\partial C_r(\mathcal{Y}_n),\partial S)\big) 
	= B\big(\M,R_1+d_H(\partial C_r(\mathcal{Y}_n),\partial S)\big)\setminus \mathring B\big(\M,R_1-d_H(\partial C_r(\mathcal{Y}_n),\partial S)\big),
	\end{equation}
	then, there exists $z_n\in \partial C_r(\mathcal{Y}_n)\cap (x_n,y_n)$, $(x_n,y_n)$ being
	the open segment joining $x_n$ and $y_n$,  but then by \eqref{th4inclu2},  $d(z_n,\partial S)>d_H(\partial C_r(\mathcal{Y}_n),\partial S)$ which is a contradiction;
	this concludes the proof of \eqref{th4inclu}.

	 Now we can prove $i)$. Suppose that $\mathring{\M}=\emptyset$.  
	Then $R_1=\max_{x\in S}d(x,\partial S)=\max_{x\in\M}d(x,\partial S)=d_H(\M,\partial S)$. 
	Also, as $C_r({\mathcal Y}_n)\subset S$ thus 
	\begin{equation}\label{th4ineq1}
	\tilde R_n\leq R_1.
	\end{equation}
	
      For every observation $Y_i$ let $m_i$ denote its projection on $\M$; by \eqref{th4inclu} we have 
	$d(m_i,\partial C_r(\mathcal{Y}_n))\geq R_1-d_H(\partial C_r(\mathcal{Y}_n),\partial S)$ so that, from triangular inequality,
	 
	$d(m_i,Y_i)+d(Y_i,\partial C_r(\mathcal{Y}_n))\geq  R_1-d_H(\partial C_r(\mathcal{Y}_n),\partial S)$. Thus
	\begin{equation}\label{th4ineq2}
	\tilde{R}_n\geq R_1-d_H(\partial C_r(\mathcal{Y}_n),\partial S)-\min_id(Y_i,\M).
	\end{equation}
	
	We now analyze the order of the last term in \eqref{th4ineq2}. From the assumption of finiteness of the Minkowski content of $\M$, given a constant $A>0$ there exists a constant $c_{\M}>0$
	such that for $n$ large enough, 
	
	$$
	\mu_d\left(\mathcal{B}\big(\M,(A\log(n)/n)^{1/(d-d')}\big)\right)\geq c_\M A\log(n)/n.
	$$
	Thus, 
	$$P_X \big(\forall i, d(Y_i,\M)\geq  (A\log(n)/n)^{1/(d-d')}\big)\leq \big(1-f_0 c_\M A(\log(n)/n)\big)^n\leq n^{-f_0c_\M A}.$$
	If we take $A>1/(f_0c_\M)$ we obtain, from Borel-Cantelli lemma, 
	\begin{equation}\label{th4ineq3}
	\min_i d(Y_i,\M)={\mathcal O}\left((\log(n)/n)^{1/(d-d')}\right),\ \mbox{a.s.}
	\end{equation}
	
	Finally, \eqref{ineqthbis} is a direct consequence of \eqref{eqthrhull}, \eqref{th4ineq1}, \eqref{th4ineq2} and \eqref{th4ineq3}.
	
	The proof of $ii)$ is obtained as in Theorem \ref{Th:noisy1} part $ii)$.
\end{proof}

\end{theorem}

\begin{remark}
	The assumption imposed on $\M$ in part (i) can be seen as an statement of $d^\prime$-dimension\-ality. For example if we assume that $\M$ is rectifiable then, from Theorem 3.2.39 in \cite{fed:69}, the $d'$-dimensional Hausdorff measure of $\M$, ${\mathcal H}^{d^\prime}(\M)$ coincides with the corresponding Minkowski content. Hence 
	$0<{\mathcal H}^{d^\prime}(\M)<\infty$ and, according to expression \eqref{Def-dimhaus}, 
	this entails $\mbox{dim}_H(\M)=d^\prime$. 
\end{remark}

\subsection{An index of closeness to lower dimensionality}
According to Theorem \ref{Th:noisy1} in the case $R_1=0$, the value $2\hat{R}_n/\widehat{\diam}(\M)$ (where $\widehat{\diam}(\M)=\max_{i\neq j}\break \|X_i-X_j\|$) can
be seen as an index of departure from low-dimensionality. Observe that if $\M=\overline{\mathring{\M}}$ we get $2\hat R_n/\widehat{\diam}(\M)\to 1$, a.s. and if $\M$ has empty interior, $2\hat R_n/\widehat{\diam}(\M)\to 0$ a.s.

 \section{A method to partially denoise the sample data} \label{sec:alg}

There are several situations in which we may speak of ``noise in the data'': we could first mention the ``outlier model'' in which the noise is given by a certain amount of outlying observations, far away from the central core of the data. Also, we might have a situation in which every observation is perturbed with a small amount of noise. We will present in this section a denoising proposal, dealing with the latter case and related to the models considered in the previous sections. Before presenting  this proposal we will let us briefly comment some references that, from different points of view, deal with the problem of noisy samples in geometric/statistical contexts.

Sometimes the term ``denoise'' is replaced with ``declutter'' in the literature on stochastic geometry. A general ``declutter algorithm'', depending on a single parameter has bee recently proposed in \cite{buc15}. This paper includes also a short interesting overview of the literature on the topic. In particular, the authors mention two main general declutter methodologies, namely procedures based on deconvolution (where the distribution generating the noise appears convolved with the ``true'' underlying model), see \cite{cai13}, and those based on thresholding, \cite{oze11}, where the data are ``cleaned'' using an auxiliary density estimator.  

Another interesting approach to the denoising idea, different to that followed in this paper, is given in \cite{chazal:11}. These authors tackle the identification of some geometric or topological features from samples that could include outliers. Again, they use the $r$-offsets (that is the $r$-parallel sets of the sample data and the target set $S$) as a fundamental tool. Such $r$-offsets are represented in terms of sublevel sets of appropriate functions, defined as a short of distance between a point and a set. The main contribution in the mentioned paper is to robustify (against outliers) such distance functions, and the corresponding sublevel sets, by replacing them with a new function that can be seen as a distance between a point and a probability distribution. 
A recent related approach, based on the use of kernel density estimates, can be found in \cite{phill:2015}.

The denoising idea is also alike to that of identifying (from a sample of points on the set $S$) the ``central part'' of the set, often called ``skeleton'' or ``medial axis'' of $S$. See \cite{cue14}  and references therein. In fact, the possible idea of defining a denoising procedure in terms of distance to the medial axis, could be seen as a sort of ``dual'' version of the method proposed in the present paper, based on the distance to the estimated boundary. 
 
Closely related ideas, ultimately relying on the notion of medial axis, are considered in \cite{dey15}, where a method for ``sparsification'' of a sample is proposed. The aim is also (as in the denoising case) to retain a subset of the original sample, which is assumed to be drawn on a manifold. In authors' words: \textit{``We sparsify the data so that the resulting set is locally uniform and is still good for homology
inference''}.  The proposed method is based on the ``lean feature size'' distance, which is intermediate between the well-know ``local feature size'' (defined in terms of the medial axis) and the ``weak local feature size''.
\subsection{The algorithm}\label{subsec:algorithm}

 Let $\M\subset \mathbb{R}^d$ be a compact set with ${\rm reach}(\M)=R_0>0$. Let $\mathcal{Y}_n=\{Y_1,\dots,Y_n\}$ be an iid  sample of a random variable $Y$, with absolutely continuous distribution whose support is the parallel set $S=B(\M,R_1)$ for some $0<R_1<R_0$.    We now propose an algorithm to get from  
$\mathcal{Y}_n$, a ``partially de-noised'' sample of points ${\mathcal Z}_m$ that allow us to estimate the target set $\M$, as established in Theorem \ref{thalg}.

The procedure works as follows:

\begin{enumerate}
	\item \textit{Take suitable  auxiliary estimators for $S$ and $R_1$}. Let $\hat{S}_n$ be an estimator of $S$ (based on $\mathcal{Y}_n$) such that $d_H(\partial \hat{S}_n,\partial S)<a_n$ eventually a.s., for some $a_n\rightarrow 0$. Let $\hat{R}_n$ be an estimator of $R_1$ such that $|\hat{R}_n-R_1|\leq e_n$ eventually a.s. for some $e_n\rightarrow 0$.
	\item \textit{Select a $\lambda$-subsample far from the estimated boundary of $S$}. Take $\lambda\in (0,1)$ and define $\mathcal{Y}^\lambda_m=\{Y^\lambda_1,\dots,Y^\lambda_m\}\subset \mathcal{Y}_n$ where $Y^\lambda_i\in \mathcal{Y}_m^\lambda$ if and only if $d(Y^\lambda_i,\partial \hat{S}_n)>\lambda \hat{R}_n$. 
	\item \textit{The projection + translation stage}. 
	For every $Y_i^\lambda\in \mathcal{Y}_m^\lambda$, we define $\mathcal{Z}_m=\{Z_1,\dots,Z_m\}$ as follows,
	\begin{equation} \label{denoise}
	Z_i=\pi_{\partial \hat{S}_n}(Y_i^\lambda)+\hat{R}_n\frac{Y_i^\lambda-\pi_{\partial \hat{S}_n}(Y_i^\lambda)}{\|Y_i^\lambda-\pi_{\partial \hat{S}_n}(Y_i^\lambda)\|},
	\end{equation}
	where $\pi_{\partial \hat{S}_n}(Y_i^\lambda)$ denotes the metric projection of $Y_i^\lambda$ on $\partial \hat{S}_n$.
\end{enumerate}
 
  \subsection{Asymptotics} 
The following result shows that the above de-noising procedure allows us to asymptotically recover the ``inner set'' $\M$.

\begin{theorem} \label{thalg}  Let $\M\subset \mathbb{R}^d$ be a compact set with ${\rm reach}(\M)=R_0>0$. Let $\mathcal{Y}_n=\{Y_1,\dots,Y_n\}$ be an iid sample of $Y$, with support $S=B(\M,R_1)$ for some $0<R_1<R_0$, and distribution $P_Y$, absolutely continuous with respect to the Lebesgue measure,   whose density $f$, is bounded from below by $f_0>0$.  Let $a_n$ and $e_n$ be, respectively, the convergence rates in the estimation of $\partial S$, as defined in the algorithm of Sunsection \ref{subsec:algorithm}.    Then, there exists $b_n=\mathcal{O}\left(\max(a_n^{1/3},e_n,\eps_n)\right)$ such that, with probability one,
for $n$ large enough,
	{$$d_H(\mathcal{Z}_m,\M)\leq b_n$$ }
where $\eps_n=c(\log(n)/n)^{1/d}$ with $c>(6/(f_0\omega_d))^{1/d}$ and $\mathcal{Z}_m$ denotes the denoised sample defined in the algorithm.  
\end{theorem}
	\begin{proof} 
	 First let us prove that    $d_H(\mathcal{Y}_n,S)\leq \eps_n$ eventually a.s.. To do that, we will use Theorem 4 in \cite{crc:04} as it was done in Theorem \ref{Th:noisy1}. 
	 By Corollary 4.9 in \cite{fed59}, $\text{reach}(\overline{S^c})>0$ and then by Proposition \ref{prop-stand}, $S$ is standard. Again by Corollary 4.9
	 in \cite{fed59} $\text{reach}(\partial S)>0$, which entails, by Proposition \ref{lembea} that $S$ fulfils the outside rolling condition. Using Theorem 4 and Proposition 1 in \cite{crc:04} 
	 we conclude that,   $d_H(\mathcal{Y}_n,S)\leq \eps_n$ eventually a.s.. 
	 
	  Let us fix $Y_i^\lambda \in \mathcal{Y}_m^\lambda$.
	
	Let us denote $l=\|Y_i^\lambda-\pi_{\partial S}(Y_i^\lambda)\|$ and $\eta_i=(Y_i^\lambda-\pi_{\partial S}(Y_i^\lambda))/l$, let us
	introduce two estimators $\hat{l}=\|Y_i^\lambda-\pi_{\partial \hat{S}_n}(Y_i^\lambda)\|$ and
	$\hat{\eta}_i=(Y_i^\lambda-\pi_{\partial \hat{S}_n}(Y_i^\lambda))/\hat{l}$. With this notation $Z_i=\pi_{\partial \hat{S}_n}(Y_i^\lambda)+\hat{R}_n\hat{\eta}_i$. 
	Recall that since ${\rm reach(\M)}>R_1$ we have (by Corollary 4.9  in \cite{fed59}) that $\pi_{\M}(Y_i^\lambda)=\pi_{\partial S}(Y_i^\lambda)+R_1\eta_i$,\\
	
	For all $Y_i^\lambda$ there exists a point $x\in \partial \hat{S}_n$ with $||x-\pi_{\partial S}(Y_i^{\lambda})||\leq a_n$
	so that, by triangular inequality:  $d(Y_{i}^{\lambda},\partial \hat{S}_n)\leq l+a_n$ that is, 
	\begin{equation}\label{inclu1}
	 \pi_{\partial \hat{S}_n}(Y_i^\lambda)\in \mathcal{B}(Y_i^\lambda,l+a_n).
	\end{equation}
	
	Now let us prove that 
	\begin{equation}\label{inclu2}
	 \pi_{\partial\hat{S}_n}(Y_i^\lambda) \in \mathcal{B}(Y_i^\lambda,l-a_n)^c.
	\end{equation}
	Suppose by contradiction that  $\pi_{\partial\hat{S}_n}(Y_i^\lambda) \in \mathcal{B}(Y_i^\lambda,l-a_n)$, 
	since $d_H(\partial S_n,\partial S)<a_n$ there exists $t\in \partial S$ such that $\|t-\pi_{\partial\hat{S}_n}(Y_i^\lambda)\|< a_n$,
	but then $l=d(Y_i^\lambda,\partial S)\leq \|Y_i^\lambda-\pi_{\partial\hat{S}_n}(Y_i^\lambda)\|+
	\|\pi_{\partial\hat{S}_n}(Y_i^\lambda)-t\|<l$. That concludes the proof of \eqref{inclu2}.
	
	By \eqref{inclu1} and \eqref{inclu2} we have:
	\begin{equation}\label{indegllchap}
	l-a_n\leq \hat{l}\leq l+a_n.
	\end{equation}
    In the same way it can be proved that  
	\begin{equation}\label{inclu3}
	\pi_{\partial\hat{S}_n}(Y_i^\lambda) \in \mathcal{B}(\pi_\M(Y_i^\lambda),R_1-a_n)^c.
	\end{equation}	
	Let us prove that there exists $C_0>0$ such that 
	\begin{equation}\label{thalgstep1}
	 \text{ for all }
	 Y_i^\lambda\in \mathcal{Y}^\lambda_m \text{ , }\|Z_i-\pi_\M(Y_i^\lambda)\|\leq C_0\sqrt{a_n^{2/3}+a_n^{1/3}e_n+e_n^2}.
	\end{equation}
	
	First consider the case $0\leq R_1-l\leq a_n^{1/3}$, which implies that $\|Y_i^{\lambda}-\pi_\M(Y_i^{\lambda})\|\leq a_n^{1/3}$. 
	Notice that, by \eqref{indegllchap}, $\|Y_i^{\lambda}-Z_i\|=|\hat{R}_n-\hat{l}|\leq a_n^{1/3}+e_n+a_n$, 
	finally we get
	\begin{equation}\label{distcase1}
	 \|Z_i-\pi_\M(Y_i^{\lambda})\|\leq 2a_n^{1/3}+a_n+e_n.
	\end{equation}	
	Now we consider the case $R_1-l\geq a_n^{1/3}$, recall that by \eqref{inclu1} and \eqref{inclu3}
	we have.
	
	\begin{equation} \label{eq1th4}
	\pi_{\partial \hat{S}_n}(Y_i^\lambda)\in \mathcal{B}(Y_i^\lambda,l+a_n)\setminus \mathcal{B}(\pi_\M(Y_1^\lambda),R_1-a_n).
	\end{equation}
		
	In Figure \ref{fignoise} it is represented the case for which $\|\pi_{\partial \hat{S}_n}(Y_i^\lambda)-\pi_{\partial S}(Y_i^\lambda)\|$ takes 
	its largest possible value. 
		
	\begin{figure}[!h]
		\centering
		\includegraphics[scale=1]{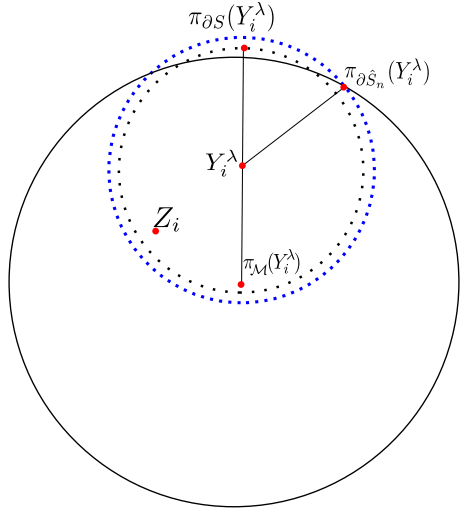}
		\caption{In solid black line $\mathcal{B}(\pi_\M(Y_1^\lambda),R_1-a_n)$, in dashed line $\mathcal{B}(Y_i^\lambda,\hat{l})$ and $\mathcal{B}(Y_i^\lambda,l)$.}%
		\label{fignoise}
	\end{figure}
	 To find an upper bound for such value, let us first note that the points $\pi_{\partial \hat{S}_n}(Y_i^\lambda)$, $Y_i^\lambda$ and $\pi_{\partial \hat{S}_n}(Y_i^\lambda)+R_1\hat{\eta}_i$
	are aligned. And the points	$\pi_{\partial S}(Y_i^\lambda)$, $Y_i^\lambda$ and $\pi_\mathcal{M}(Y_i^\lambda)$,  are aligned. So all of them are
	in the same plane $\Pi$.  Let us now apply a translation T in order to get, 
	$T(\pi_{\partial S}(Y_1^\lambda))=0$. Let us consider in $\Pi$ a coordinate system $(x,y)$ such that $\pi_{\M}(Y_i^\lambda)=(0,-R_1)$. 
	
	Let $(x_1,y_1)$ be the coordinates of the point $\pi_{\partial \hat{S}_n}(Y_i^\lambda)$. From \eqref{eq1th4}  we get
	\begin{align}
	x_1^2+(y_1+l)^2&\leq (l+a_n)^2 \label{ineq1} \\
	x_1^2+(y_1+R_1)^2 &\geq (R_1-a_n)^2\label{ineq3}
	\end{align}
	
	If we multiply \eqref{ineq3} by $-l$, we get $-l(x_1^2+y_1^2)-2y_1lR_1\leq -la_n^2+2a_nlR_1$ and if we multiply \eqref{ineq1} 
	by $R_1$ we get $R_1(x_1^2+y_1^2)+2y_1lR_1\leq 2R_1a_nl+a_n^2R_1$.
	Then, if we sum this two inequalities we get,
	\begin{equation} \label{eqalg1}
	x_1^2+y_1^2\leq \frac{4lR_1}{R_1-l}a_n+a_n^2 \leq 4R_1^2a_n^{2/3}+a_n^2.
	\end{equation}
	 Notice that $Z_i\in \Pi$,  let us denote 
	 $(x,y)$ the coordinates of $Z_i$ in $\Pi$, then $$x=x_1-\hat{R}_n\frac{x_1}{\|Y_i^\lambda-\pi_{\partial \hat{S}_n}(Y_i^\lambda)\|}=x_1-\hat{R}_n\frac{x_1}{\hat{l}}$$
	and
	$$y=y_1-\hat{R}_n\frac{l+y_1}{\|Y_i^\lambda-\pi_{\partial \hat{S}_n}(Y_i^\lambda)\|}=y_1-\hat{R}_n\frac{l+y_1}{\hat{l}}.$$
	 Since the coordinates of $\pi_\M(Y^\lambda_i)$ are $(0,-R_1)$ we get that 
	  \begin{multline} \label{eqalg2}
	   \|Z_i-\pi_\M(Y^\lambda_i)\|^2=\|(x,y+R_1)\|^2=(x_1^2+y_1^2)\left(\frac{\hat{l}-\hat{R}_n}{\hat{l}}\right)^2+\left(R_1-\hat{R}_n\frac{l}{\hat{l}}\right)^2\\
	   +2y_1\left|\frac{\hat{l}-\hat{R}_n}{\hat{l}}\right|\left(R_1-\hat{R}_n\frac{l}{\hat{l}}\right).
	  \end{multline}   
	   Observe that $|l-\hat{l}|\leq a_n$ and $|R_1-\hat{R}_n|\leq e_n$ eventually almost surely. We can bound 
	    $|\frac{\hat{l}-\hat{R}_n}{\hat{l}}|\leq 2$ and 
	   $\hat{l}\geq \lambda R_1/2$, then 
	 \begin{equation}\label{eqalg2.0}
	    \left|R_1-\hat{R}_n\frac{l}{\hat{l}}\right|=\left|\frac{R_1(\hat{l}-l)-l(\hat{R}_n-R_1)}{\hat{l}}\right|\leq 2\frac{a_n+e_n}{\lambda}.
	 \end{equation}
	 Finally by equations \eqref{eqalg1}, \eqref{eqalg2} and \eqref{eqalg2.0}, 
	if $R_1-l\geq a_n^{1/3}$  (note that this is used in the proof of \eqref{eqalg1}), there exists $C_0$ such that 

	\begin{equation} \label{eqalg3}
	\|Z_i-\pi_\M(Y^\lambda_i)\|\leq C_0\sqrt{a_n^{2/3}+a_n^{1/3}e_n+e_n^2}.
	\end{equation}
	where (see \eqref{eqalg1}) we are using $y_1=\mathcal{O}(a_n^{1/3})$ here.
	That concludes the proof of \eqref{thalgstep1}.\\

	Let us finally prove that $\M\subset B(\mathcal{Z}_m,a_n+e_n+2\eps_n)$ eventually, a.s. 
	As indicated at the beginning of the proof,  we have $d_H(\mathcal{Y}_n,S)\leq \eps_n$ eventually a.s., thus for all $x\in \M$, there exists $Y_i\in \mathcal{Y}_n$ such that
	$\|x-Y_i\|\leq \eps_n$. For $n$ large enough we have $Y_i\in \mathcal{Y}_m^{\lambda}$. Following the same ideas used to prove \eqref{distcase1} we obtain
	$\|Z_i-Y_i \|\leq \eps_n+a_n+e_n$. 
	By triangular inequality we get
	\begin{equation}\label{eqalg4}
	 \M\subset B(\mathcal{Z}_m,a_n+e_n+2\eps_n)\text{ eventually, a.s.}
	\end{equation}
	Combining \eqref{distcase1}, \eqref{eqalg3} and \eqref{eqalg4} we obtain,
	$$d_H(\mathcal{Z}_m,\M) = 
	\mathcal{O}\left(\max(a_n^{1/3},e_n,\sqrt{a_n^{2/3}+a_n^{1/3}e_n+e_n^2},\eps_n)\right)=\mathcal{O}\left(\max(a_n^{1/3},e_n,\eps_n)\right).$$	
	
	\end{proof}

\begin{remark} Note that, when $\mathring{\M}=\emptyset$,  the result simplifies since, according to Theorem \ref{Th:noisy1} we can take $e_n=2\epsilon_n$ and, according to \cite{crc:04}  (Prop. 1 and Th. 4) $a_n=\epsilon_n$. Therefore, in this case $b_n=a_n^{1/3}$.
\end{remark}

The two following corollaries give the exact convergence rate for the denoising process introduced before, using the centers of the boundary balls (Corollary \ref{denoiseBB}), and the boundary of the $r$-convex hull (Corollary \ref{denoiseCr}), as estimators of the boundary of the support.
\begin{corollary}\label{denoiseBB} 
Let $\M\subset \mathbb{R}^d$ be a compact set such that $\emph{reach}(\M)=R_0>0$.
Let  $\mathcal{Y}_n=\{Y_1,\dots,Y_n\}$ be an iid sample of a distribution 
$P_Y$ with support $B(\M,R_1)$ for some $0<R_1<R_0$. Assume that $P_Y$ is 
absolutely continuous with respect to the Lebesgue measure and the density $f$, 
is bounded from below by a constant $f_0>0$.
Let $\eps_n=c(\log(n)/n)^{1/d}$ and  $c>(6/(f_0\omega_d))^{1/d}$.

Given $\lambda\in (0,1)$, let  $\mathcal{Z}_n$ be the points obtained after the 
denoising process using $\hat{R}_n$ to estimate
$R_1$ and $\{Y_i,i\in I_{bb}\}$ as an estimator of $\partial S$ where
 $I_{bb}=\{j:\mathcal{B}(Y_j,\eps_n) \text{ is a boundary ball}\}$. Then,
$$d_H(\mathcal{Z}_m,\M)=\mathcal{O}\big((\log(n)/n)^{1/(3d)}\big),\ \mbox{a.s.}$$
\end{corollary}
 
Using the assumption of $r$-convexity for $\M$ (see Definitions  \ref{def-rconvexity} and \ref{def-reach} and the subsequent comments) in the construction of the set estimator, we can replace $\hat R_n$ with $\tilde R_n$ (see Theorem \ref{thnoise1}). Then, at the cost of some additional complexity in the numerical implementation, a faster convergence rate can be obtained. This is made explicit in the following result.  
\begin{corollary}\label{denoiseCr} 
Let $\M\subset \mathbb{R}^d$ be a compact $d'$-dimensional  set (in the sense of Theorem \ref{thnoise1}, i)  such that $\emph{reach}(\M)=R_0>0$.
Let  $\mathcal{Y}_n=\{Y_1,\dots,Y_n\}$ be an iid sample of a distribution 
$P_Y$ with support $B(\M,R_1)$ for some $0<R_1<R_0$. Assume that $P_Y$ is 
absolutely continuous with respect to the Lebesgue measure and the density $f$, 
is bounded from below by a constant $f_0>0$.

For a given $\lambda\in (0,1)$, let $\mathcal{Z}_n$ be the set of the points obtained after 
the denoising process, based on the
estimator $\partial C_r(\mathcal{Y}_n)$ of $\partial S$ (for some $r$ with $0<r<\min(R_0-R_1,R_1)$) and the
estimator $\tilde{R}_n$ of $R_1$.

Then,

$$d_H(\mathcal{Z}_m,\M)=\mathcal{O}\big((\log(n)/n)^{2/(3(d+1))}\big)  \mbox{ a.s.} $$
 
 \end{corollary}

\section{Estimation of lower-dimensional measures}\label{sec:minkowski}

\subsection{Noiseless model}

In this section, we go back to the noiseless model, that is, we assume that the sample points $X_1,\ldots,X_n$ are drawn according to a distribution whose support is $\M$. The target is to estimate the $d'$-dimensional Minkowski content of $\M$, as given by
\begin{equation} \label{minkest}
\lim_{\epsilon\rightarrow 0} \frac{\mu_d(B(\M,\epsilon))}{\omega_{d-d'}\epsilon^{d-d'}}=L_{0}(\M)<\infty.
\end{equation}
This is just (alongside with Hausdorff measure, among others) one of the possible ways to measure lower-dimensional sets; see \cite{mat95} for background.   

In recent years, the problem of estimating the $d'$-dimensional measures of a compact set from a random sample has received some attention in the literature. The simplest situation corresponds to the full-dimensional case $d'=d$. Any estimator $\M_n$ of $\M$ consistent with respect to the \textit{distance in measure}, that is $\mu_d(\M_n\Delta \M) \to 0$ (in prob. or a.s., where $\Delta$ stands for the symmetric difference), will provide a consistent estimator for $\mu_d(\M)$.  In fact, as a consequence of Th. 1 in \cite{dw:80}
(recall that $S$ is compact here) this will the always the case (in probability) when $\M_n$ is the \textit{offset} estimator (\ref{dw}), provided that $\mu_d$ is absolutely continuous (on $\M$)  with respect to $P_X$ together with $r_n\to 0$ and $nr_n^d\to\infty$. 

Other more specific estimators of $\mu_d(\M)$ can be obtained by imposing some shape assumptions on $\M$, such as convexity or  $r$-convexity, which are incorporated to the estimator $\M_n$; see \cite{ari16}, \cite{bal15},  \cite{par11}.

Regarding the estimation of lower-dimensional measures, with $d'<d$, the available literature mostly concerns the problem of estimating $L_0(\M)$, $\M$ being the boundary of some compact support $S$. The sample model is also a bit different, as it is assumed that we have sample points \textit{inside and outside} $S$. Here, typically, $d'=d-1$; see \cite{cue07}, \cite{cue13}, \cite{jim11}.  

Again, in the case $\M=\partial S$ with $d=2$, under the extra assumption of $r$-convexity for $S$, the consistency of the plug-in estimator $L_0(\partial C_r({\mathcal X}_n))$ of $L_0(\partial S)$ is proved in \cite{cue12} under the usual \textit{inside} model (points taken on $S$). Finally, in \cite{ber14}, assuming an \textit{outside} model (points drawn in $B(S,R)\setminus S$), estimators of $\mu_d(S)$ and $L_0(\partial S)$ are proposed, under the condition of \textit{polynomial volume} for $S$ 

From the perspective of the above references, our contribution here (Th. \ref{consistency} below) could be seen as a sort of lower-dimensional extension of the mentioned results  of type $\mu_d(\M_n)\to\mu_d(\M)$ regarding volume estimation. But, obviously, in this case the Lebesgue measure $\mu_d$ must be replaced with a lower-dimensional counterpart, such as the Minkowski content (\ref{minkest}). We will also need the following lower-dimensional version of the standardness property given in Definition \ref{estandar}.

\begin{definition} \label{stand} A Borel probability measure 
	 defined on a $d'$-dimensional set $\M\subset \mathbb{R}^d$ (considered with the topology induced by $\mathbb{R}^d$) is said to be standard with respect to the $d'$-dimensional Lebesgue measure $\mu_{d'}$ if there exist $\lambda$ and $\delta$ such that, for all $x\in \M$,
	$$P_X(\mathcal{B}(x,r))=\mathbb{P}(X\in \mathcal{B}(x,r)\cap \M)\geq \delta \mu_{d'}(\mathcal{B}(x,r)), \ \mbox{\it for}\ 0\leq r\leq \lambda.$$
\end{definition}

\begin{remark} \label{remstand} Observe that, by Lemma 5.3 in (\cite{sma08}) this condition is fulfilled if $P_X$ has a density $f$ bounded from below and $\M$ is a manifold with positive \textit{condition number} (also known as positive reach). Standardness of the distribution has also been used in \cite{cue04}, \cite{chazal:13}, \cite{aam15}. 
\end{remark}

\begin{theorem} \label{consistency} Let $\mathcal{X}_n=\{X_1,\dots,X_n\}$ be an iid sample drawn according to a distribution $P_X$ on a set $\M\subset \mathbb{R}^d$. Let us assume that the distribution $P_X$ is standard with respect to the $d'$-dimensional Lebesgue measure (see \ref{stand}) and that there exists the $d'$ Minkowski content $L_0(\M)<\infty$ of $\M$, given by (\ref{minkest}).
	Let us take $r_n$ such that $r_n\rightarrow 0$ and $(\log(n)/n)^{1/d'}=o(r_n)$, then
	\begin{itemize}
		\item[(i)] 		\begin{equation} \label{minkest2}
		\lim_{n\rightarrow \infty}\frac{ \mu_d \big(B(\mathcal{X}_n,r_n)\big)}{\omega_{d-d'}r_n^{d-d'}}=L_0(\M)\quad a.s.
		\end{equation}	
		\item[(ii)]  If $\emph{reach}(\M)=R_0>0$, then
	$$\frac{ \mu_d \big(B(\mathcal{X}_n,r_n)\big)}{\omega_{d-d'}r_n^{d-d'}}-L_0(\M)=\mathcal{O}\Big(\frac{\beta_n}{r_n}+r_n\Big),$$
		where $\beta_n=\mathcal{O}\big(\log(n)/n\big)^{1/d'}$.	
	\end{itemize}

	\begin{proof} (i) First we will see that, following the same ideas as in Theorem 3 in \cite{crc:04} it can be readily proved that, with probability one, for $n$ large enough, \
		\begin{equation} \label{orddh}
		d_H(\mathcal{X}_n,\M)\leq C\beta_n.
		\end{equation}	
		for some large enough constant $C>0$.	 In order to see (\ref{orddh}), let us consider $M_\Delta$ a minimal covering of $\M$, with balls of radius $\Delta$ centred in $N_\Delta$ points belonging to $\M$. Let us prove that $N_\Delta=\mathcal{O}(\Delta^{-d'})$. Indeed, since $M_\Delta$ is a minimal covering it is clear that $\mu_d(B(\M,\Delta))\geq N_\Delta \omega_d (\Delta/2)^d$, and then 
		
		$$\frac{\mu_d(B(\M,\Delta))}{\omega_{d-d'}\Delta^{d-d'}}\geq \frac{ N_\Delta \omega_d(\Delta/2)^{d}}{\omega_{d-d'}\Delta^{d-d'}}=c_1 N_\Delta \Delta^{d'},$$
		$c_1$ being a positive constant. Since there exists $L_0(\M)$ it follows that $N_\Delta=\mathcal{O}(\Delta^{-d'})$. Then the proof of (\ref{orddh}) follows easily from the standardness of $P_X$ and $N_\Delta=\mathcal{O}(\Delta^{-d'})$, so we will omit it.
		
		Now, in order to prove (\ref{minkest2}), let us first prove that, if we take  $\alpha_n=1-C\beta_n/r_n$ ,
		\begin{equation} \label{eqconsistency}
		B(\M, \alpha_n r_n) \subset B(\mathcal{X}_n , r_n)\subset B(\M,r_n) \quad a.s..
		\end{equation}
		To prove this, consider $x_n\in B(\M, \alpha_n r_n)$, then there exists $t_n\in \M$ such that $x_n\in \mathcal{B}(t_n,\alpha_n r_n)$. Since $d_H(\mathcal{X}_n, \M)\leq C\beta_n$ there exists  $y_n\in \mathcal{B}(t_n,C\beta_n)$, $y_n\in \mathcal{X}_n$. It is enough to prove that $x_n\in \mathcal{B}(y_n,r_n)$. But this follows from the fact that, eventually a.s.,
		 
		\begin{equation*}
		\|y_n-x_n\|\leq \|x_n-t_n\|+\|t_n-y_n\|\leq \alpha_n r_n+C\beta_n= r_n.
		\end{equation*}
		 Then, from \eqref{eqconsistency}
		
		\begin{equation} \label{eqth4}
		\alpha_n^{d-d'} \frac{\mu\big(B(\M, \alpha_n r_n)\big)}{\omega_{d-d'} \alpha_n^{d-d'} r_n^{d-d'}} - L_0(\M) \leq
		\frac{\mu \big(B(\mathcal{X}_n, r_n)\big)}{\omega_{d-d'}r_n^{d-d'}}-L_0(\M)\leq \frac{\mu\big(B(\M, r_n)\big)}{\omega_{d-d'} r_n^{d-d'}}-L_0(\M).
		\end{equation}
		Since there exists $L_0(\M)$, the right hand side of \eqref{eqth4} goes to zero.
		To prove that the left hand side of \eqref{eqth4} goes to zero, let us observe that, as $\alpha_n=1-C\beta_n/r_n$, and $\alpha_n^{d-d'}=1-\mathcal{O}(\beta_n/r_n)$, then
		\begin{equation} \label{eq2th4}
		\alpha_n^{d-d'} \frac{\mu\big(B(\M, \alpha_n r_n)\big)}{\omega_{d-d'} \alpha_n^{d-d'} r_n^{d-d'}} - L_0(\M)=\frac{\mu\big(B(\M, \alpha_n r_n)\big)}{\omega_{d-d'} \alpha_n^{d-d'} r_n^{d-d'}}-\mathcal{O}(\beta_n/r_n)-L_0(\M),
		\end{equation}
		since $\alpha_n\rightarrow 1$ and $\beta_n/r_n\rightarrow 0$ we get 
		$$\lim_{n\rightarrow \infty} \frac{\mu\big(B(\M, \alpha_n r_n)\big)}{\omega_{d-d'} r_n^{d-d'}}=L_0(\M)\quad a.s..$$
		(ii)  The assumption $\emph{reach}(\M)=R_0>0$ allow us to ensure that $\M$ has a polynomial volume in the interval $[0,r_0)$. This means that, for $r<R_0$, $\mu \big(B(\M,r)\big)=P_d(r)$ where $P_d(r)$ is a polynomial of degree at most $d$; this is a classical result due to \citet[Th. 5.6]{fed59}. Since we assume that the $d'$-Minkowski content $L_0(\M)$ is finite, this polynomial volume condition entails that the coefficient to the $d-d'$ term is $\omega_{d-d'}L_0(\M)$. Then,
		$$\frac{\mu(B(\M,r_n))}{\omega_{d-d'}r_n^{d-d'}}=L_0(\M)+r_n A(\M)+o(r_n),$$ 
		for some constant $A(\M)$.
		Now the proof follows from \eqref{eqth4} and \eqref{eq2th4}.
	\end{proof}
\end{theorem}

\begin{remark} In the case of sets with positive reach, part (b) 
suggests to take $r_n^2 =\max_i\min_{j\neq i}\|X_i-X_j\|$ since we know by Theorem 1 in
\cite{penrose:99} that $r_n^2=\mathcal{O}\big((\log(n)/n)^{1/d'})$  that gives the optimal convergence rate.
\end{remark}
\subsection{Noisy Model}

The estimation of the Minkowski content in the noisy model has been tackled in \cite{ber14}, where the random sample is assumed to have uniform distribution in the parallel set $U$.  
In this section we will see that even if the sample is not uniformly distributed on $B(\M,R_1)$ for some $0<R_1<R_0={\rm reach}(\M)$, it is still possible, by applying first the de-noising algorithm introduced in Section \ref{sec:alg}, to estimate $L_0(\M)$.
Following the notation in Section \ref{sec:alg}, let $\mathcal{Y}_n$ be an iid sample of a random variable $Y$ with support $B(\M,R_1)$, let us denote $\mathcal{Z}_m$ the de-noised sample defined by \eqref{denoise}. The estimator is defined as in \eqref{minkest2} but replacing $\mathcal{X}_n$ with  $\mathcal{Z}_m$. Although the subset $\mathcal{Z}_m$ is not an iid sample (since the random variables $Z_i$ are not independent), the consistency is based on the fact that $\mathcal{Z}_m$ converge in Hausdorff distance to $\M$, as we will prove in the following theorem.

\begin{theorem} With the hypothesis and notation of Theorem \ref{thalg}, if $\max(a_n^{1/3},e_n,\eps_n)=o(r_n)$ where $\eps_n=c(\log(n)/n)^{1/d}$  with $c>(6/(f_0\omega_d))^{1/d}$. Then,
		\begin{equation} \label{minknoise}
		\lim_{n\rightarrow \infty} \frac{\mu_{d}(B(\mathcal{Z}_m,r_n))}{\omega_{d-d'}r_n^{d-d'}}=L_0(\M)\quad a.s..
		\end{equation}
\end{theorem}
\begin{proof}The proof is analogous to the one in Theorem \ref{consistency}. Observe that in Theorem \ref{thalg} we proved that $d_H(\mathcal{Z}_m,\M)\leq b_n$, for some $b_n=\mathcal{O}(\max(a_n^{1/3},e_n,\eps_n))$, then $b_n/r_n\rightarrow 0$. As we did Theorem \ref{consistency} if we take $\alpha_n=1-b_n/r_n$, then, with probability one, 
	\begin{equation*}
	B(\M, \alpha_n r_n) \subset B(\mathcal{Z}_m , r_n) \subset B(\M,r_n),
	\end{equation*} then we get
		\begin{equation*}
		\alpha_n^{d-d'} \frac{\mu\big(B(\M, \alpha_n r_n)\big)}{\omega_{d-d'} \alpha_n^{d-d'} r_n^{d-d'}} - L_0(\M) \leq
		\frac{\mu \big(B(\mathcal{Z}_n, r_n)\big)}{\omega_{d-d'}r_n^{d-d'}}-L_0(\M)\leq \frac{\mu\big(B(\M, r_n)\big)}{\omega_{d-d'} r_n^{d-d'}}-L_0(\M),
		\end{equation*}
		from where it follows 
		\begin{equation*}
		\alpha_n^{d-d'} \frac{\mu\big(B(\M, \alpha_n r_n)\big)}{\omega_{d-d'} \alpha_n^{d-d'} r_n^{d-d'}} - L_0(\M)=\frac{\mu\big(B(\M, \alpha_n r_n)\big)}{\omega_{d-d'} \alpha_n^{d-d'} r_n^{d-d'}}-\mathcal{O}(b_n/r_n)-L_0(\M).
		\end{equation*}
		 Since $\alpha_n\rightarrow 1$ and $b_n/r_n\rightarrow 0$ we get \eqref{minknoise}. 
	\end{proof}
	
\section{Computational aspects and simulations}\label{sec:comput}

We discuss here some theoretical and practical aspects regarding the implementation of the algorithms. We present also some simulations and numerical examples. 

\subsection{Identifying the boundary balls}
The cornerstone of the practical use of Theorem 1 is the effective identification of the boundary balls. The following proposition provides the basis for such identification, in terms of the Voronoi cells of the sample points.  
Recall that, given a finite set $\{x_1,\dots,x_n\}$, the Voronoi cell associated with the point $x_i$ is defined by ${\rm \text{Vor}}(x_i)=\{x:d(x,x_i)\leq d(x,x_j) \text{ for all } i\neq j\}$.
\begin{proposition} \label{Prop:propalg} 
	Let $\mathcal{X}_n=\{X_1,\ldots,X_n\}$ be an $iid$ sample of points, in $\mathbb{R}^d$, drawn according to a distribution $P_X$, absolutely continuous with respect to 
	the Lebesgue measure. Then, with probability one, for all $i=1,\dots,n$ and all $r>0$, $\sup\{\|z-X_i\|,z\in \emph{Vor}(X_i)\}\geq r$ if and only if $\mathcal{B}(X_i,r)$  
	is a boundary ball for the Devroye-Wise estimator (\ref{dw}).
	
	\begin{proof} Let us take $r>0$ and $X_i$ such that there exists $z\in \partial \mathcal{B}(X_i,r)\cap \mbox{Vor}(X_i)\neq \emptyset$, 
		let us prove that $z\in \partial \hat{S}_n(r)$. Observe that since $z\in \mbox{Vor}(X_i)$, $d(z,\mathcal{X}_n\setminus X_i)\geq r$  thus $d(z,\mathcal{X}_n)=r$. 
		Reasoning by contradiction suppose that $z\in \mathring{\hat{S}}_n$ then, with probability one, there exists $j_0$ such that $z\in\mathring{\mathcal{B}}(X_{j_0},r)$ and so
		$\|z-X_{j_0}\|<r$ that is a contradiction.
		
		Now to prove the converse implication let us assume that $\mathcal{B}(X_i,r)$ is a boundary ball, 
		then there exists $z\in \partial\mathcal{B}(X_i,r)$ such that $z\in \partial \hat{S}_n(r)$. Let us prove that $d(z,\mathcal{X}_n\setminus X_i)\geq r$ 
		(from where it follows that $z \in \mbox{Vor}(X_i)$ ). Suppose that $d(z,\mathcal{X}_n\setminus X_i)<r$, then there exists $X_j\neq X_i$ such that $d(z,X_j)<r$ and 
		then $\mathcal{B}\big(z,r-d(z,X_j)\big)\subset \mathring{\hat{S}}_n(r)$.
	\end{proof}
\end{proposition} 

\subsection{An algorithm to detect empty interior in the noiseless case using Theorem \ref{Th:noiseless}}

In order to use in practice Theorem 1 to detect lower-dimensionality in the noiseless case, we need to fix a sequence $r_n\downarrow 0$ under the conditions indicated in  Theorem \ref{Th:noiseless} (ii).  Note that this requires to assume lower bounds for the ``thickness'' constant $\rho({\mathcal M})=\sup d(x,\partial{\mathcal M})$ and the  standardness constant $\delta$ (see Definition \ref{def-stand}) as well as an upper bound for the radius of the outer rolling ball. 

Now, according to Theorem \ref{Th:noiseless}, and Proposition \ref{Prop:propalg}, we will use the following algorithm. 
\begin{itemize}
	\item[1)] For $i=1,\dots,n$, let $V^i=\{V_1^i,\ldots V_{k_i}^i\}$ be the vertices of $\text{Vor}(X_i)$,
	\item[2)] Let $\delta_i=\sup\{\|z-X_i\|,z\in \text{Vor}(X_i)\}=\max\{\|X_i-V_k^i\|,1\leq k\leq k_i\}$, since $\text{Vor}(X_i)$ is a convex polyhedron. In the case that
	$\text{Vor}(X_i)$ is an unbounded cell we put $\delta_i=\infty$. Define $\delta_0=\min_i \delta_i$.
	\item[3)] Decide $\mathring{\M}\neq \emptyset$ if and only if $\delta_0 \geq r_n$.
\end{itemize}

\subsection{On the estimation of the maximum distance to the boundary}
Theorems \ref{Th:noisy1} and \ref{thnoise1} involve the calculation of quantities such as $d(x,\partial \hat{S}_n(\epsilon_n))$ and $d(x,\partial C_r({\mathcal Y}_n))$, where $\hat{S}_n(\epsilon_n)$ is a Devroye-Wise estimator of type (\ref{dw}) and $C_r({\mathcal Y}_n)$ is the $r$-convex hull \eqref{rhull} of ${\mathcal Y}_n$. 

It is somewhat surprising to note that, in spite of the much simpler structure of $\hat{S}_n(\epsilon_n)$ when compared to $C_r({\mathcal Y}_n)$, the distance to the boundary $d(x,\partial C_r({\mathcal Y}_n))$ can be calculated in a simpler, more accurate way than the analogous quantity $d(x,\partial \hat{S}_n(\epsilon_n))$ for the Devroye-Wise estimator $\hat{S}_n(\epsilon_n))$. 

Indeed note that $d(x,\partial C_r(\mathcal{Y}_n))$ is relatively simple to calculate; this is done in \cite{ber12} in the two-dimensional case although can be in fact used in any dimension. 
Observe first that $\partial C_r(\mathcal{Y}_n))$ is included in a finite union of spheres of radius $r$, with centres in $Z=\{z_1,\dots,z_m\}$. Then $d(x,\partial C_r(\mathcal{Y}_n))=\min_{z_i\in Z}\|x-z_i\|-r$.  In order to find $Z$ we need to compute the Delaunay triangulation. Recall that the Delaunay triangulation, $\text{Del}(\mathcal{Y}_n)$, is defined as follows. Let $\tau \subset \mathcal{Y}_n$,
$$\tau \in \text{Del}(\mathcal{Y}_n) \quad \text{ if and only if} \quad  \bigcap_{Y_i\in \tau} \text{Vor}(Y_i)\neq \emptyset.$$
Observe finally, for any dimension, $\bigcap_{Y_i\in \tau} \text{Vor}(Y_i)\neq \emptyset$ is a segment or a half line. If $\tau_i$ is the $d$-dimensional simplex with vertices $\{Y_{i_1},\ldots,Y_{i_{d}}\}\subset \partial \mathcal{B}(z_i,r)$, the point $z_i$ can be obtained as $\bigcap_{Y^i_j\in \tau_i} \text{Vor}(Y_i)\cap \mathcal{B}(Y^i_1,r)$.

\subsection{Experiments}

The general aim of these experiments is not to make an extensive, systematic empirical study. We are just trying to show that the methods and algorithm proposed here can be implemented in practice. 

\

\noindent \textit{Detection of full dimensionality}.
 We consider here a simple illustration of the use of Theorem 1 and the associated algorithm.  In each case, we draw 200 samples of sizes $n=$ 
50, 100, 200, 300, 400, 500, 1000, 2000, 5000, 10000 on the 
$A$-parallel set around  the unit sphere, $\partial \mathcal{B}(0,1)\subset\mathbb{R}^d$;  that is, the sample data are selected  on 
 $\mathcal{B}(0,1+A)\setminus \mathring{\mathcal{B}}(0,1-A)$.  The width parameter $A$ takes the values 
$A=0,0.01,0.05,0.1,\ldots,0.05$. 
Table \ref{table:sphere} provides the minimum sample sizes to ``safely decide'' the correct answer. This means to correctly decide on, at least 190 out of 200 considered samples, that the support is lower dimensional (in the case $A=0$)
or that it  is full dimensional (cases with $A>0$).

We have used the boundary balls procedure (here and in the denoising experiment below for $A=0$) with $r=2\max_i(\min_{j\neq i}\Vert X_j-X_i\Vert)$.

The results look quite reasonable: the larger the dimension $d$ and the smaller the width parameter $A$, the harder the detection problem. 

\begin{table}[ht]
	\begin{center}	
		\begin{tabular}{c|ccc}
			$A$	    & $d=2$ & $d=3$ & $d=4$ \\
			\cline{2-4} 
			0	    & $\leq 50$ & $\leq 50$ & $\leq 50$ \\ 
			0.01	& $[51,100]$ & $[1001, 2000]$  & $>10000$ \\ 
			0.05	& $\leq 50$  & $[201, 300]$ & $[1001, 2000]$ \\ 
			0.1	    & $\leq 50$  & $[51, 100]$ & $[101, 200]$ \\ 
			0.2	    & $\leq 50$  & $\leq 50$ & $[51, 100]$ \\ 
			0.3	    & $\leq 50$  & $\leq 50$ & $[51, 100]$  \\ 
			0.4	    & $\leq 50$  & $\leq 50$  & $\leq 50$  \\ 
			0.5	    & $\leq 50$  & $\leq 50$  & $\leq 50$
		\end{tabular}
		\footnotesize\caption{Minimum sample sizes required to detect 
			lower dimensionality for different values of the dimension $d$ and the width parameter $A$.}\label{table:sphere}
	\end{center}
\end{table}

\

\noindent \textit{Denoising}.
We draw points on  
$\mathcal{B}(0,1.3)\setminus\mathring{\mathcal{B}}(0,0.7)$  in $\mathbb{R}^2$ and $\mathbb{R}^3$.

In order to evaluate the effectiveness of the denoising procedure we define  the random variable $e=\Vert Y\Vert-1$ from the denoised data $Y$ and also from the original data. Note that the ``perfect'' denoising would correspond to $e=0$. The Figure \ref{figdensity} shows the kernel
estimators of both densities of $e$ for the case $d=2$ (left panel) and for $d=3$ (right panel). These estimators for the denoised case are based on $m=100$ values of $e$ extracted from samples of sizes $n=$100, 1000, 10000. The density estimators for the initial distribution are based on samples of size 100.  Clearly, when the denoised sample of size $m=100$ is based on a very large sample, with $n=10000$, the denoising process is better, as suggested by the fact that the corresponding density estimators are strongly concentrated around 0. The slight asymmetry in the three dimensional case, accounts for the fact that the ``external'' volume $\mathcal{B}(0,1.3)\setminus\mathcal{B}(0,1)$ is larger than the ``internal'' one 
$\mathcal{B}(0,1)\setminus\mathcal{B}(0,0.7)$.

\begin{figure}[ht]
	\begin{center}
		\includegraphics[height=5cm,width=7cm]{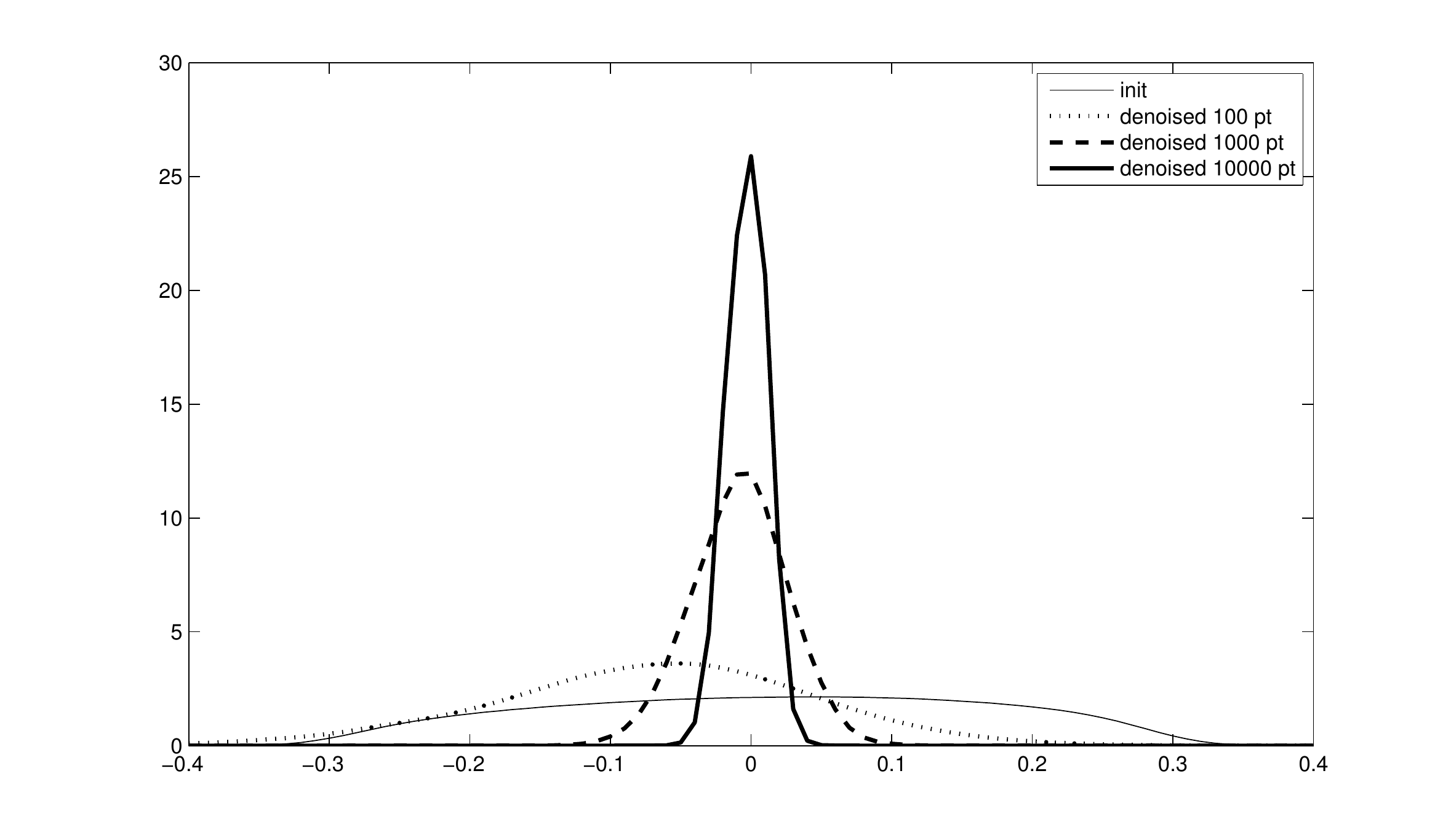}
		\includegraphics[height=5cm,width=7cm]{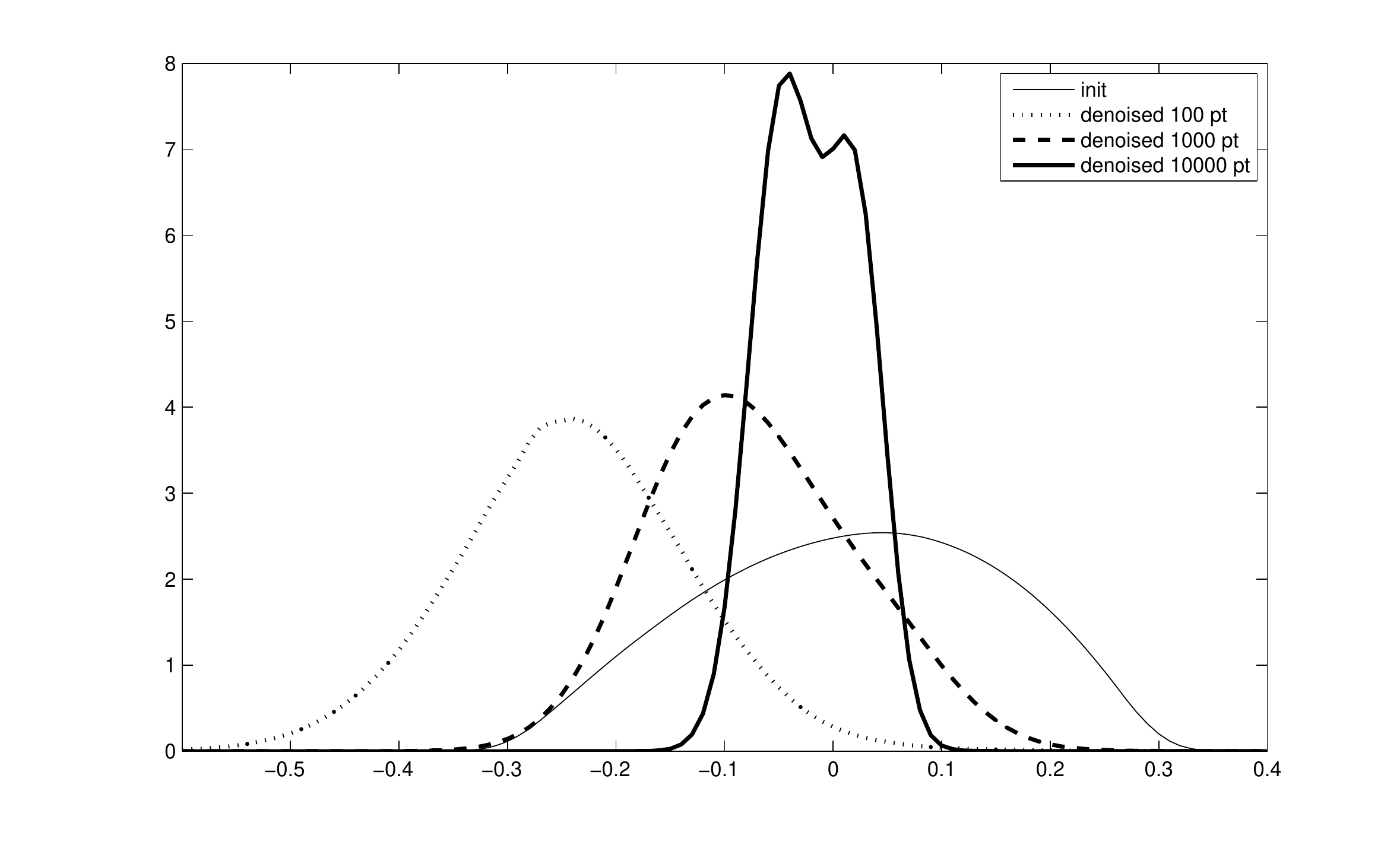}
	\end{center}
	\caption{Estimated density functions of the random variable $e=\Vert Y\Vert-1$ for $d=2$ (left) and $d=3$ right, with and without denoising.}
	\label{figdensity}
\end{figure}

Figures \ref{figrsquare} and \ref{figtref} provide a more visual idea on the result of the denoising algorithm. They correspond, respectively,   to the set $\mathcal{B}(S_{L_3},0.3)$ (where $S_{L_3}=\{(x,y),|x|^3+|y|^3=1\}$)
and to $\mathcal{B}(T,0.3)$, where $T$ is  the so-called \textit{Trefoil Knot}, a well-known curve with interesting topological and geometric properties.

\begin{figure}[ht]
	\begin{center}
		\includegraphics[height=6cm,width=6cm]{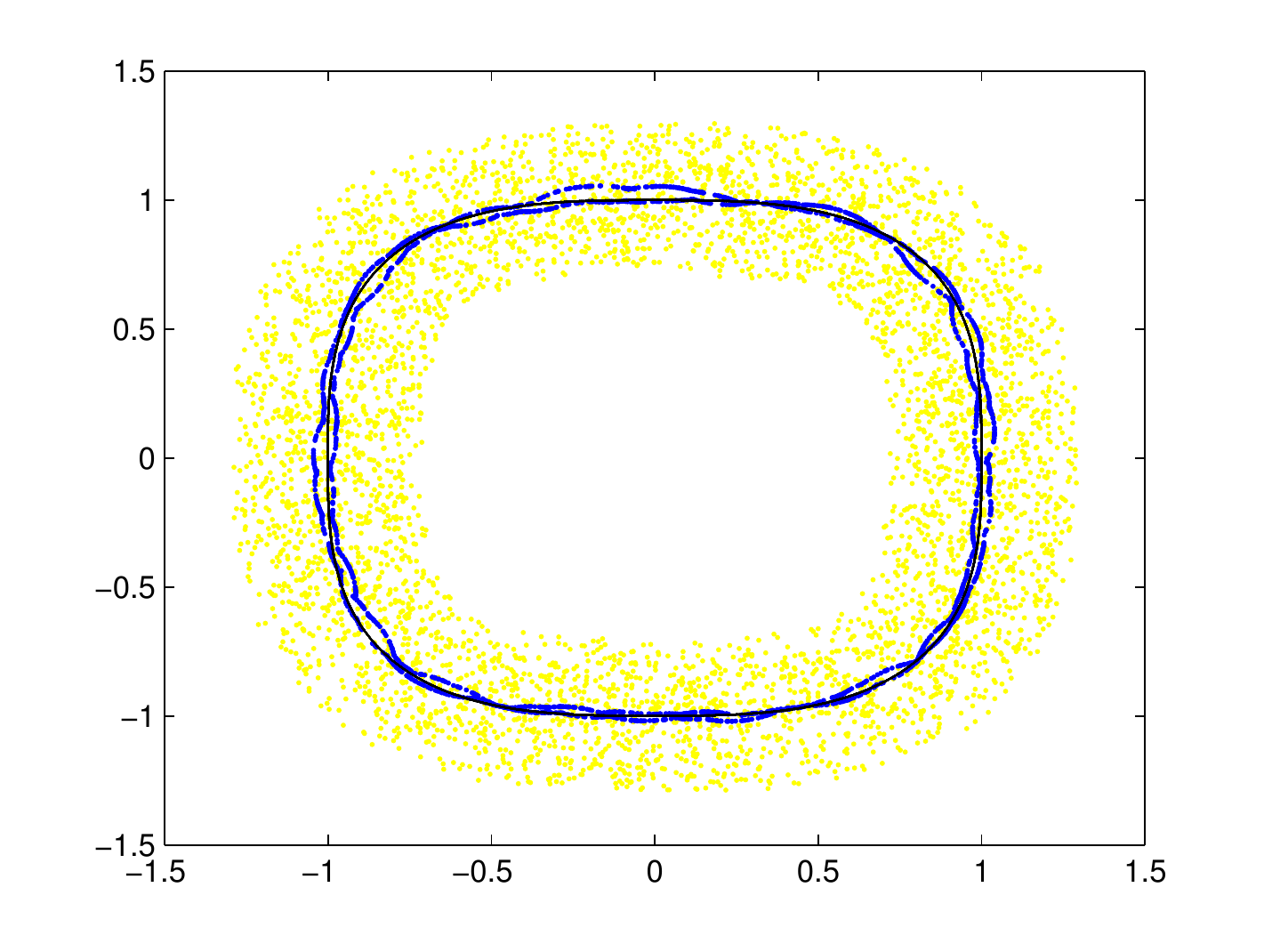}
		\includegraphics[height=6cm,width=6cm]{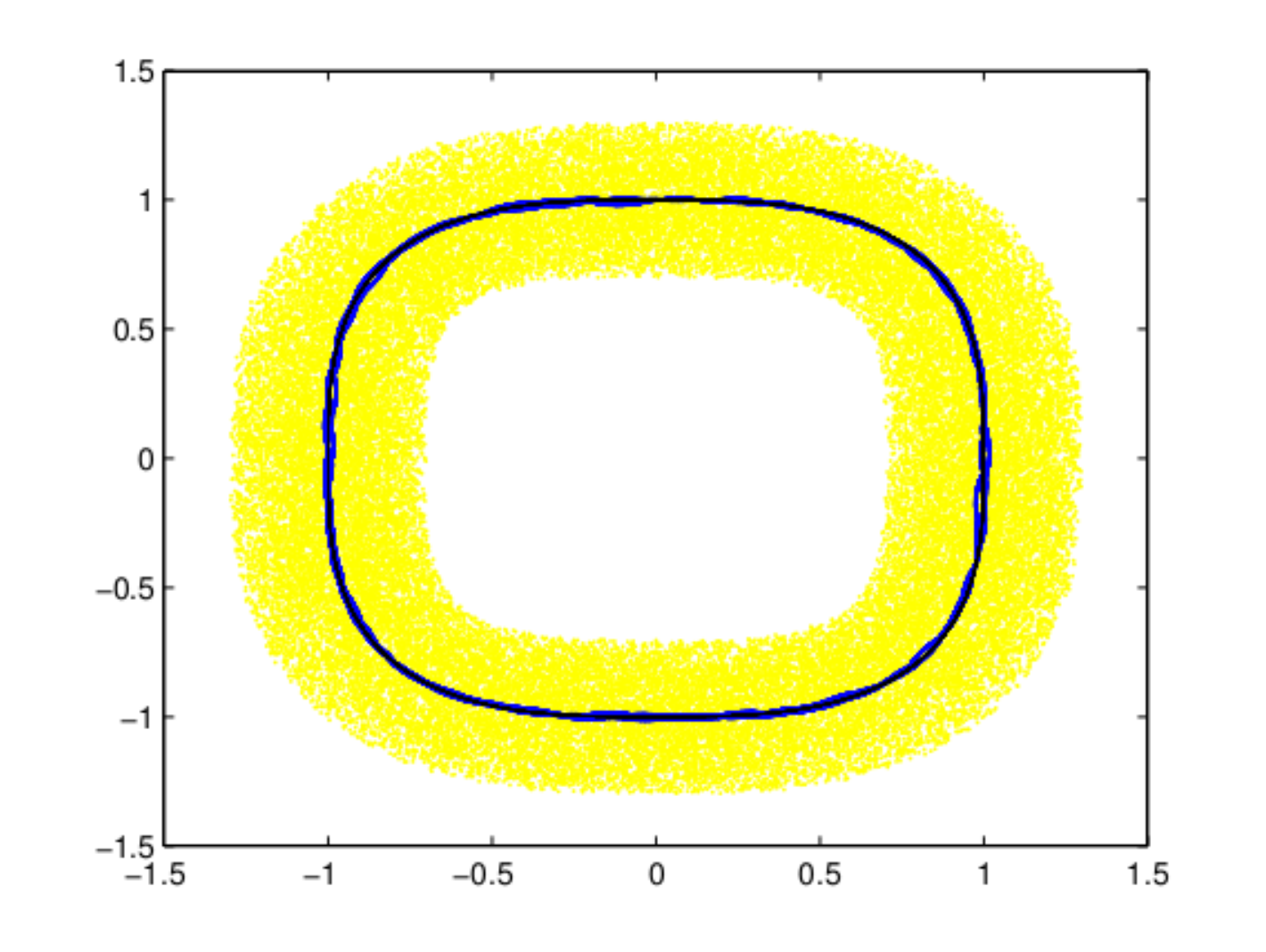}
	\end{center}
	\caption{The yellow background is made of $5000$ points (left) and $50000$ points (right) drawn on on $\mathcal{B}(S_{L_3},0.3)$, with $S_{L_3}=\{(x,y),|x|^3+|y|^3=1\}$. The blue points are the result
		of the denoising process. The black line corresponds to the original set $S_{L_3}$}
	\label{figrsquare}
\end{figure}

\begin{figure}[ht]
	\begin{center}
		\includegraphics[height=10cm,width=6cm]{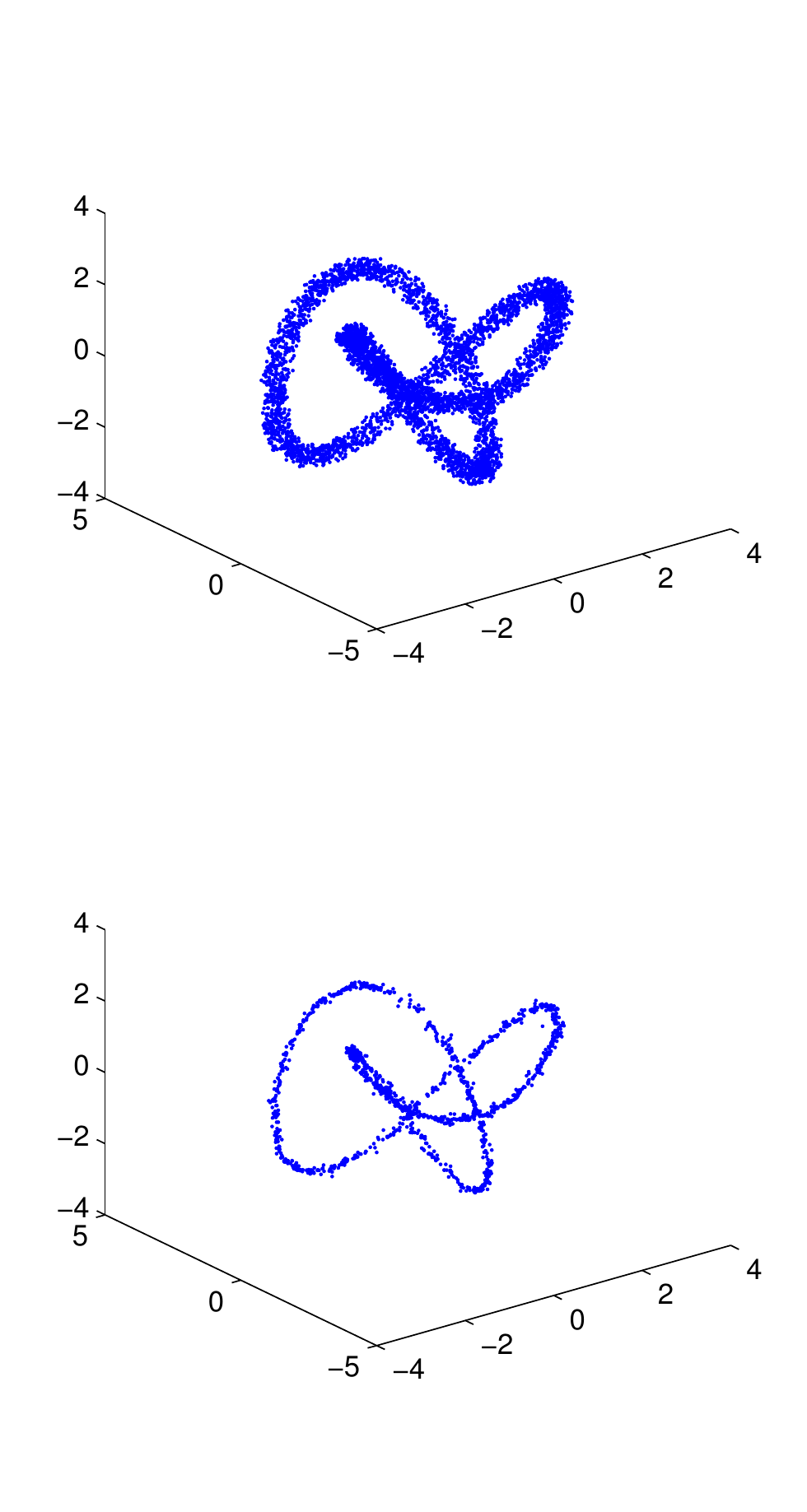}
		\includegraphics[height=10cm,width=6cm]{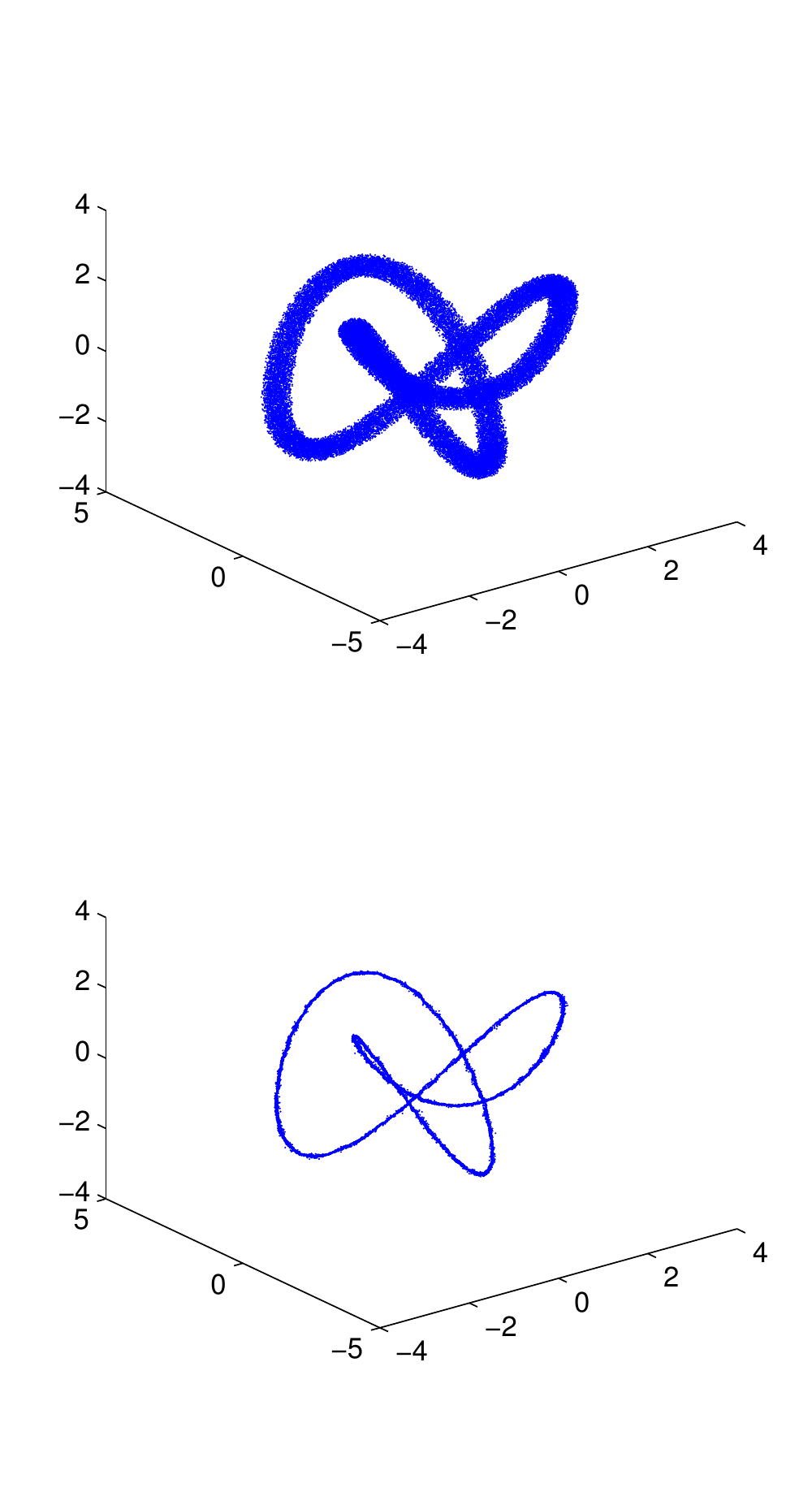}	 		
	\end{center}
	\caption{The upper panel shows $5000$ noisy points (left) and $50000$ noisy points (right) drawn on $\mathcal{B}(T,0.3)$. The lower panel shows the result of the corresponding denoising process.}
	\label{figtref}
\end{figure}

\newpage

\noindent \textit{Minkowski contents estimation}.
Finally in Table \ref{table:minkowski} we show, just as a tentative experiment, some results about the Minkowki contents estimation, again in the case of noiseless data
($R_1=0$) and noisy points (with $R_1$=0.2) drawn around a sphere  for different values for $n$
and different dimensions.  

\begin{table}[ht]
	\begin{center}	
		\begin{tabular}{|c|cccc|}
			\hline
			&$n=10^3$&$n=10^4$ & $n=10^5$ & $n=10^6$\\
			\hline
			$d=2$ &	$0.11$ 	&$0.1$	&$0.08$	&$0.07$\\
			$d=3$ &	$0.14$	&$0.14$&$0.13$&$0.12$\\
			$d=4$ &	$0.16$	&$0.16$&$0.16$&$0.15$\\
			\hline
			
		\end{tabular}
		\footnotesize\caption{Radius $r_0(n,d)$ for Minkowski contents estimation when $R_1=0.2$}
		\label{table:rad}
	\end{center}
\end{table}

\begin{table}[ht]
	\begin{center}	
		\begin{tabular}{|cc|cccc|}
			\hline
			$d$ &	$R_1$ &	 $n=10^3$&$n=10^4$ & $n=10^5$ & $n=10^6$\\
			\hline
			$2$ &	$0$  &	$0.38$ 	&$0.34$	&$0.32$	&$0.3$\\
			$2$ &	$0.2$ &	$27.29$	&$11.79$&$5.33$&$2.37$\\
			\hline
			$3$ &	$0$ &	$4$	&$0.88$	&$0.37$	&$0.32$\\
			$3$ &	$0.2$ &	$37.03$	&$28.76$&$19.95$&$11.35$\\
			\hline
			$4$ &	$0$  &	$16.1$	&$4.34$	&$1.23$	&$0.45$\\
			$4$ &	$0.2$ &	$91.37$	&$54.85$&$26.69$&$25.88$\\
			\hline
			
		\end{tabular}
		\footnotesize\caption{Relative errors (in percentage) for Minkowski contents estimation}
		\label{table:minkowski}
	\end{center}
\end{table}

For every $R_1,n,d$ we estimate the Minkowski contents 
using a radius $r=0.5\break \sqrt{\max_i (\min_{j\neq i}\|X_i-X_j\|)}$  (see Theorem \ref{th:choosern})  when $R_1=0$ and with a 
a deterministic radius $r=r_0(n,d)$ ( slowly decreasing with the dimension,  see Table \ref{table:rad}) when $R_1=0.2$. 
The values of the estimators have been calculated via a Monte Carlo Method based on $10^5$ points uniformly drawn on $\mathcal{B}(0,1+2r)\setminus \mathring{\mathcal{B}}(0,1-2r)$.
For every $R_1,n,d$ the experiment has been done $100$ times. Table \ref{table:minkowski} entries provide
the average relative error (in percentage) in the estimation of the boundary Minkowski contents $L$.
That is, the entries are $100\cdot err(R_1,d)$ where 
$err(R_1,d)=\frac{1}{L}\sqrt{\sum_i (L_{i}(R_1,d)-L)^2/100}$, $L$ being the correct value of the boundary length in each case, that is $L=2\pi$, $4\pi$, $2\pi^2$, for $d=2,3,4$, respectively.

 Even if we disregard the intrinsic difficulties associated with the Monte Carlo approximation,  the outputs of Table \ref{table:minkowski} suggest that the denoising-based methodology for the estimation of the Minkowski content from noisy observations, is not accurate for large dimensions. Note however that the problem is intrinsically difficult, as shown by the convergence rates obtained in the noiseless case. Note also that the noise level $R_1=0.2$ is quite large, especially for $d=3,4$.   In any case, the results displayed in Figure \ref{figtref} suggest a quite reasonable performance of the denoising procedure, for other descriptive or image analysis purposes. Clearly, more research would be needed to reach more definitive conclusions.

\newpage

\section*{Acknowledgements} This research has been partially supported by MATH-AmSud grant
16-MATH-05 SM-HCD-HDD (C. Aaron and A. Cholaquidis) and Spanish grant MTM2016-78751-P (A. Cuevas).  We are grateful to Luis Guijarro and Jes\'us Gonzalo (Dept. Mathematics, UAM, Madrid) for useful conversations and advice.

\end{document}